\documentclass[11pt, oneside, reqno]{book}
\usepackage{amssymb, amsthm, amsfonts}
\usepackage{amsmath,amscd}
\usepackage{bardproj}
\usepackage{graphics}
\usepackage[pdftex]{graphicx}
\usepackage{epsfig}
\usepackage{charter}
\usepackage[scaled]{helvet}
 
\newcommand{\R}{\mathbb{R}}
\newcommand{\Z}{\mathbb{Z}}

\newcommand{\C}{\mathbb{C}}
\newtheorem*{discussion}{Discussion}

\begin{document}

\titlepg{Bases for representation rings of Lie groups and their maximal tori}{Changwei Zhou}
    {May}{2012}


\abstr

\indent A {\bf{Lie group}} is a group that is also a differentiable manifold, such that the group operation is continuous respect to the topological structure. To every Lie group we can associate its tangent space in the identity point as a vector space, which is its Lie algebra. Killing and Cartan completely classified simple Lie groups into seven types.\\

Representation of a Lie group is a homomorphism from the Lie group to the automorphism group of a vector space. In general represenations of Lie group are determined by its Lie algebra and its the connected components. We may consider operations like direct sum and tensor product with respect to which the representations of $G$ form a ring $R(G)$. \\

Assume the group is compact. For every Lie group we may find a maximal torus inside of it. By projecting the representation ring of the Lie group to the representation ring of its maximal torus, we may consider to express $R(T)$ elements as sums of products of base elements with $R(G)$ elements. In 1970s Pittie and Steinberg proved $R(T)$ is a free module over $R(G)$ using homological algebra methods without specifying the bases in each seven types. \\

In this senior project I found an explicit basis for the representation ring of the maximal torus of Lie group $SU_{n}$ in terms of the represenation ring of the Lie group itself.  I also did some computation with $SO_{2n}$.\\

\tableofcontents

\dedic

To {\bf{Marina Day}}. I know you will be better. And Steven Wu. 

\acknowl

\indent Without my advisor Professor {\bf{Gregory Landweber}}'s guidance, this senior project would never have started. \\\indent An important hint was given by Professor {\bf{David Speyer}} in Mathoverflow. \\\indent Professor {\bf{James Belk}} gave valuable suggestions that helped me to complete this project when I was stuck. 

I am grateful to Professor {\bf{Sternberg}}'s lecture notes on Lie algebra which expanded my understanding on this subject. 

\begin{center}
With gratitude to the following friends:\\

\noindent Daniela Anderson\\
Jin Zhibin\\
Leah Gastler\\
Mary Ann Jackson\\
Nathan Savir\\
Oleg Lazarev\\
Simon Rubinstein-Salzedo\\
Stanley Moore\\
Steven Karp\\
Steven Wu\\
Valance Wang\\
Ying Xiaozheng\\
Zana Chan(Giang Tran:{\bf{my great, great roommate and suitemate}})\\
\end{center}

\chapter{Background and general introduction}
\label{chap }

\doublespace
\section{Brief summary}
It is well known that for a compact Lie group its representations are completely reducible \cite{Bott2}. Thus we may consider the representation rings $R(G)$ and $R(T)$ formed under formal inverse: $$E-F=G-H\leftrightarrow E\oplus H\cong G\oplus F$$ It is also well-known that $R(T)\equiv \mathbb{Z}[\Lambda]$, $R(G)\equiv \mathbb{Z}[\Lambda]^{W}$, where $\Lambda$ is the character lattice of $T$ and $W$ is the Weyl group. In other words the representation ring of the Lie group is isomorphic to the invariant subring under the action of the Weyl group in $\Z[\Lambda]$. 

We proved in this paper that $\prod x_{i}^{a_{i}},0\le a_{i}\le n-i$ is a basis for $R(T(SU_{n}))$ over $R(SU_{n})$ and discussed $n=1,2,3$ cases in $SO_{2n}$. In here $x_{i}$ are weights in $H^{*}$ such that $x_{i}(h_{j})=\delta_{ij}$, $h_{j}\in H$. Our main method is an induction proof based on the minimal polynomial satisfied by elements in $R(G)$. We also come up with different proofs for the $SU_{3}$ case and employed index theory (to be discuss below) to verify the bases we found are correct. 

\section{Historical background}
\label{section }
\indent In 1963, Atiyah and Singer published their landmark paper on the index theorem \cite{Atiyah}, which states the global analytical behavior of a differential operator can be measured by its topological invariants on the manifold. Its basic form is $$\text{dim Ker}D-\text{dim coKer}D=ch[D][td(D)](X)$$ We are going to briefly explain the meaning of the formula in Chapter 2. 

\indent Atiyah and Singer initally proved their theorem based on Gronthedieck's proof of his Grothendieck-Riemann-Roch Theorem. But their second (and later) modified proofs are remarkably different as they used K-theoretic tools \cite{Atiyah2}. This work suggested that $K$-theory would play a greater role in future development of index theory. Subsequently topological K-theory has become indispensable in advanced level analysis. On the other hand it should be possible to work on the differential operator side and reach similar generalizations via analytic methods. 

Replacing an arbitrary differentiable manifold with one built from groups, letting us use abstract algebra in the place of analysis, Raoul Bott \cite{Bott1} considered the case  with the additional structures as follows:

\indent {\bf  1}: The base manifold, $X$, is the coset space $G/H$ as a compact connected Lie group G by a closed connected subgroup $H$. \\
\indent {\bf  2}: Both $E$ and $F$ are induced by representations of $H$. \\
\indent {\bf  3}: The differential operator $D:\Gamma E\rightarrow \Gamma F$ commutes with the induced action of $G$ on $\Gamma E$ and $\Gamma F$. 

With these conditions assumed Bott considered an analog of the Atiyah-Singer index theorem in terms of representation rings using tools like the Peter-Weyl Theorem, he showed that:

\indent {\bf Theorem III}: Let the rank of $H$ equal the rank of $G$, and assume that $\pi_{1}(G)$ has no 2-torsion. Then for every $G$-module $W$, there exists a homogeneous elliptic operator $D_{W}: \Gamma E\rightarrow \Gamma F$, with $\chi(D)=[W]$. \\
\indent In here $\chi$ denotes the symbol, which is the difference between kernel and cokerl of $D$ in the representation ring. 

\indent {\bf Theorem IV}: Let rank $H=$ rank $G$, and let $D:\Gamma E\rightarrow \Gamma F$ be an arbitrary elliptic operator over $G/H$. Then the formula of Atiyah-Singer for the index of $D$ is valid. 

Theorem III effectively linked the elements in the representation ring with elliptic differential operators, and Theorem IV gives a conclusive answer to the limited version of the index problem. But Bott's paper actually proved something more in Theorem IV: the index of the differential operator is completely determined by only the topological structure of the manifold and not the choice of operator itself. This is remarkable because it implies the local analytical behavior of the differential operator cannot change its global analytical behavior limited by homogeneous space structure. 

Based on his paper and the index theorem, one may wonder from the reversed direction if the differential operator helps to tell something about the algebraic structure of the Lie group and its maximal rank subgroup. To be more specific, assume $H$ is a maximal rank subgroup in $G$ and $R(H)$, $R(G)$ be the representation ring of $H$ and $G$ respectively, It is desirable to know the relationship between $R(H)$ and $R(G)$. Without using differential operators, Harsh Pittie proved the following result in 1972\cite{Pittie}:

\indent {\bf Theorem 1.1} Let $G$ be a connected compact Lie group with $\pi_{1}G$ free and $G'$ a closed connected subgroup of maximal rank, then $R(G')$ is free (as a module) over $R(G)$ (by restriction). 

The technical restriction of $\pi_{1}G$ be free was considerably weakened by Robert Steinberg in his later paper in 1974 \cite{Steinberg}, where he presented as: 

\indent {\bf Theorem 1.2} Let $G$ be a connected Lie group and $S$ its semisimple component. Then the following conditions are equivalent:

\noindent {\bf (a)} $R(G')$ is free over $R(G)$ for every connected subgroup $G'$ of maximal rank. \\
{\bf (b)} $R(T)$ is free over $R(G)$ for some maximal torus $T$.\\
{\bf (c)} $R(G)$ is the tensor product of a polynomial algebra and a Laurent algebra. \\
{\bf (d)} $R(S)$ is a polynomial algebra. \\
{\bf (e)} $S$ is a direct product of simple groups, each simply connected of type $SO_{2r+1}$. 
\remark Since $\pi_{1}G$ is free if and only if $S$ is simply connected, because $G$ is the product of $S$ and a central torus, the equivalence of $(a)$ and $(e)$ provides the just-mentioned extension and converse of $\text{Theorem 1.1}$. 

Steinberg's proof, however, utilized homoglogical algebra and module theory that proved the statement in great generality. It is not clear from his work what the basis elements exactly are. Hence it is desirable to link Bott's work with this theorem to work backwards, namely for an element in the representation ring of $T$, construct a differential operator to realize it as sums of basis elements in terms of the representation ring of $G$. This is the starting point of my paper. 

We found it is possible to construct an inner product that uses the index on $G/H$ from equivariant K-theory. It accepts inputs from $R(T)$ and have outputs in $R(G)$. Further under this inner product the basis Steinberg found (we shall abbreviate as 'Steinberg's basis' ) is orthonormal in $SU_{3}$ case. However, the author could not verify if Steinberg's basis is orthonormal in general. What I can prove is the basis I found in $SU_{3}$ case is unimodular, and it is uncertain to me how to prove this in general for $SU_{n}$. The main proof that $R(T)$ is a free module over $R(SU_{n})$ with basis $\prod x_{i}^{a_{i}},0\le a_{i}\le n-i$ is done by purely algebraic methods. The interested reader of the general proof may jump to section 4.2.3 to read the proof directly. 

The next section will review the necessary technical background of above theorems, give appropriate definitions and proofs of background material. Then I will proceed to my research addressing this problem. 

\chapter{Index theory technical background}

\section{introduction}
\label{section }
The material in this section is fairly standard. For an introduction of general topology the reader may refer to Kelley's treatise \cite{Kelley}.
\section{General Topology}
\definition Let $\mathcal{X}$ be a topological space, we say $\mathcal{X}$ is a manifold if for any open set $\mathcal{X}_{\alpha}$ in some cover $\mathcal{A'}$, there is a homeomorphism $f_{\alpha}$ from $\mathcal{X}_{\alpha}$ to $\mathbb{R}^{n}$ for some $n$ such that $f_{\alpha}$ satisfies the following condition:
\item For $\alpha\not = \beta$, there exists transition function $\phi_{\alpha\beta}: \mathbb{R}^{n}\rightarrow \mathbb{R}^{n}$ such that $\phi_{\alpha\beta}f_{\alpha}=f_{\beta}$. 
\remark It is easy to notice that $\phi_{\beta\alpha}=\phi_{\alpha\beta}^{-1}$; $\phi_{\alpha\beta}=f_{\beta}f_{\alpha}^{-1}$. So it make sense to speak of the analytical behavior of $\phi_{\alpha\beta}$. 
\remark The manifold we defined above is called a {\bf{topological manifold}}. 
\definition Let $\mathcal{X}$ be a topological space, an {\bf{embedding}} from $\mathcal{X}$ to $\mathbb{R}^{n}$ is a diffeomorphism from $\mathcal{X}$ to $\mathbb{R}^{n}$. 
\remark Whitney proved that every smooth manifold of dimension $n$ admits an embedding in $\mathbb{R}^{2n}$. 
\\

\section{vector bundles and $K(\mathcal{X})$}
This section follows from \cite{Atiyah2} and \cite{Hatcher}
\begin{definition}
Let $\mathcal{X}$ be a smooth manifold. A {\bf{vector bundle}} on $\mathcal{X}$ is defined to be a map $\pi$ between a topological space $\mathcal{E}$ and $\mathcal{X}$ such that $\pi^{-1}(x)$ has a real vector space structure. The following triviality condition is required: There is a covering of $\mathcal{X}$ by open sets $U_{\alpha}$ such that $\pi^{-1}(U_{\alpha})$ is homeomorphic to $U_{\alpha}\times \R^{n}$ by some $h_{\alpha}$, and $\pi^{-1}(x)$ is homeomorphic to $x\times R^{n},\forall x\in U_{\alpha}$ by $h_{\alpha}$. \\
\indent $h_{\alpha}$ is called a {\bf{local trivialization}}. $\pi$ is called the {\bf{projection}}. The space $\mathcal{E}$ is called the {\bf{total space}}, $\mathcal{B}$ the {\bf{base space}}, and $\pi^{-1}(x)$ is called the {\bf{fiber}} we denote by $\mathcal{E}_{x}$ or $F$. If we replace $\R$ by $\C$ we yield a {\bf{complex bundle}}.
\end{definition}
\remark In this project we are working with complex bundles. 
\remark The {\bf{trivial}} vector bundle is $\mathcal{X}\times \mathbb{C}^{n}$. 
\begin{definition}
For $U_{\mathcal{X}}$ in $\mathcal{X}$, a {\bf{section}} is a map $g: \mathcal{X}\rightarrow \mathcal{E}$ that satisfies $\pi(g(U_{\mathcal{X}}))=U_{\mathcal{X}}$. We denote the set of such maps as $\Gamma E$. 
\end{definition}
\remark Since locally $U_{\mathcal{X}}$ is identified with $U_{\mathcal{E}}\cong U_{\mathcal{X}}\times \mathbb{C}^{n}$, a section on $U_{\mathcal{X}}$ is essentially a map $U_{\mathcal{X}}\rightarrow \mathbb{C}^{n}$. 
\definition Let $\mathcal{X}$ be a smooth manifold. At every point $x\in \mathcal{X}$, we may define the {\bf{tangent space}} at $x$ to be the set of tangent vectors of a curve passing through $x$.  We denote it by $T\mathcal{X}_{x}$. 
\remark Because the vectors has $n$-dimensional freedom when it is tangent to a surface of $n$ dimensions, in general (for smooth manifolds) the tangent space of a point is of the same dimension as the manifold. 
\definition Consider the topological space $\mathcal{E}=\{T\mathcal{X}_{x},x\},x\in \mathcal{X}$. We denote it by {\bf{TX}}. The projection map is provided by $(a,b)\rightarrow b$. The local trivilization associates the inverse of an open subset in $\mathcal{X}$ to $\R^{n}\times \R^{n}$, where the first $n$ is the dimension of $\mathcal{X}$ and second $n$ is the dimension of the tangent space of $\mathcal{X}$ at $x\in \mathcal{X}$. We denote this as the {\bf{tangent bundle}} of $\mathcal{X}$. 
\definition For a vector space $V$, we define its {\bf{dual space}} as the vector space generated by the linear maps from $V$ to $k$. 
\remark For a general $n$-dimensional vector space $V$ with basis $\{e_{i}\}$, its dual space is $n$-dimension as well, whose basis are $\{e_{i}^{*}\}$ as linear map assigning value $\delta_{ij}$. 
\definition The dual space for the tangent space at $x$ is called the {\bf{cotangent space}}. 
\remark We define the union of cotangent spaces using contruction we did above as the {\bf{cotangent bundle}}. We denote it by $T^{*}(\mathcal{X})$. 
\definition We define the {\bf{direct sum}} of two vector bundles over the same manifold $\mathcal{X}$ fibrewise (that is, the $V$ corresponding to $x$ in $E$). We define $[V\oplus W]: (V+W)_{x}=V_{x}\oplus W_{x}$. It is easy to verify the transition maps work properly. We also define the {\bf{tensor product}} of two vector bundles $V\otimes W$ be a new bundle: $[V\otimes W]_{x}=V_{x}\otimes W_{x}$. 
\definition A {\bf{isomorphism}} between $\mathcal{E}$ and $\mathcal{F}$ over the same base space $\mathcal{X}$ is a homeomorphism between $E$ and $F$ that commute with projections and is fibrewise linear. 
\remark This is the special case of a general morphism between vector bundles, which we will define in the `Index theorem' section. In general if a vector bundle is isomorphic to the trivial bundle, then we say it is {\bf{trivializable}}. 
\definition Consider $\oplus$ as the group operation with formal inverse $-$ such that $E-F=G-H\Leftrightarrow E\oplus H\cong G\oplus F$. The abelian group generated by the formal difference elements is called ${\bf{K(\mathcal{X})}}$. 
\example
As an example, the {\bf{differential forms}} on $\mathcal{X}$ is a smooth section from $\mathcal{X}$ to the exterior tensor product of its cotangent bundle. In other words a $n$ dimensional differential form is a anti-symmetric map $$\beta_{x}:  \prod^{n}_{1}T_{x}\rightarrow k$$ The space of degree $n$ differential forms is often denoted as $\Omega^{n}(M)$. 
\definition Consider two topological spaces $Y$ and $\mathcal{X}$ such that there is a surjective map $f: Y\rightarrow \mathcal{X}$. Assume $\mathcal{E}$ is a vector bundle over $\mathcal{X}$, we define the {\bf{pull back}} of $\mathcal{E}$ be $f^{*}\mathcal{E}$: $$f^{*}\mathcal{E}: (e,y), e\in \mathcal{E}, y\in Y, f(y)=\pi(e)$$


\section{The Analytical Index}
This section follows from \cite{E&M}
\begin{definition}
\indent Let $C^{\infty}$ denote the class of smooth functions defined on $\mathbb{R}^{n}$.  The ring of $C^{\infty}$ differentiable functions can be mapped to itself via standard differentiation:
$$\frac{\partial}{\partial x_{i}}: f\rightarrow \frac{\partial}{\partial x_{i}}(f)$$
Consider the algebra generated by $\frac{\partial}{\partial x_{i}}$ over the base field $\mathbb{R}$ or $\mathbb{C}$. We define this algebra as linear differential operators. The linear space of differential operators on $U$ of order at most $k$ is denoted by $D_{k}(U)$. 
\end{definition}
\definition
For $f\in D_{k}(U)$, its $\bf{symbol}$ is defined as a polynomial substituting the $\frac{\partial}{\partial x_{i}}$ top degree form in the linear differential operator with $i\epsilon_{i}$. The result is a homogeneous polynomial in $n$ variables of the following form:

$$\sigma(D)(x,\epsilon)=\sum_{\sum a_{i}\le k} f_{a_{0}a_{1}..a_{n}}(x)(i\epsilon_{0})^{a_{0}}\times...\times (i\epsilon_{n})^{a_{n}}$$

The {\bf{principal symbol}} of order $k$ of $D$ is the polynomial given by $$\sigma(D)_{k}(x,\epsilon)=\sum_{\sum a_{i}=k}f_{a_{0}a_{1}..a_{n}}(x)(i\epsilon_{0})^{a_{0}}\times....\times (i\epsilon_{n})^{a_{n}}$$
\remark To define the analytical index we need to consider the more general case of differential operators between sections of vector bundles. For $E$ and $F$ both defined on the base space $\mathcal{X}$, the {\bf{symbol}} of a differential operator $D$ that maps $f: C^{\infty}\in E_{x}$ to $Df\in C^{\infty}F_{x}$ is a map in the cotangent bundle of $\mathcal{X}$. We shall consider the following case where $E$ and $F$ are both trivial bundles isomorphic to $\mathcal{X}\times \mathbb{C}^{n}$. 
\remark Following above, the situation that $D$ is from $C^{\infty}[\mathbb{C}^{n}]$ to $C^{\infty}[\mathbb{C}^{n}]$ can be generalized to $\sigma(D):T^{*}(\mathcal{X})\rightarrow \mathbb{C}$ maps $\sum \epsilon_{i}dx^{i}$ to $\sigma(D)(x,\epsilon)\in \mathbb{C}$. 
\definition
A linear operator of order at most $k$ from $E$ to $F$ on $U$ is a linear map $P$: $\Gamma(U,E)\rightarrow \Gamma(U,F)$ of the form $$D=\sum_{\sum a_{i}\le k} C_{a_{0}..a_{n}} \partial_{a_{0}}\partial_{a_{1}}...\partial_{a_{n}}$$ 
We introduce the {\bf{multi-symbol}} $\alpha=(a_{0}....a_{n})$ to simplify our notation. In the above definition we have $C_{\alpha}\in \Gamma(U, Hom(E,F))$. The space of such differential operators is denoted by $D_{k}(U,E,F)$.  
\remark If for any coordinate chart $U_{k}$ on $\mathcal{X}$ with the property that $E_{U_{k}}$ is trivializable, $D$ limited at $U_{k}$ is in $D_{k}(U_{k},E,F)$, then we denote such class of differential operator by $D_{k}(E,F)$. 
\remark We denote by $\pi: T^{*}\mathcal{X}\rightarrow \mathcal{X}$ the projection map from the cotangent bundle to the manifold. Hence we may form the pull back $\pi^{*}(E)$ and $\pi^{*}(F)$ whose fibre above $\epsilon_{x}\in T^{*}(\mathcal{X})$ is $E_{x}$ and $F_{x}$ respectively. We now consider the vector bundle $Hom(E,F)$ over $\mathcal{X}$, whose fibre above $x\in \mathcal{X}$ is $Hom(E_{x},F_{x})$. Its pull back $\pi^{*}Hom(E,F)\cong Hom(\pi^{*}(E),\pi^{*}(F))$, whose fiber above $\epsilon_{x}\in T_{x}^{*}\mathcal{X}$ is $Hom(E_{x},F_{x})$. 
\definition We now redefine the principle symbol in the vector bundle case. Let $E,F$ by smooth vector bundles on $\mathcal{X}$ and let $D\in D_{k}(E,F)$. The {\bf{principal symbol}} of $D$ is a smooth function $\epsilon_{x}\in T^{*}_{x}\rightarrow \sigma_{k}(D)(\epsilon_{x})\in Hom(E_{x},F_{x})$. For each $x\in \mathcal{X}$, the function $\epsilon\rightarrow \sigma_{k}(P)(x,\epsilon)$ is a degree $k$ homogenesous polynomial function $T^{*}_{x}\mathcal{X}\rightarrow Hom(E_{x},F_{x})$. 
We omit the lengthy discussion on the analgous case of defining the polynomial explicitly. 
\example Consider the most trivial example, letting the base space as a point $x\in X$ with $E\cong \mathbb{C}^{m}$, $F\cong \mathbb{C}^{n}$. Then the section of $E,F$ is isomorphic to $E,F$. The differential operator $D$ thus degenerate to a linear homomorphism between $E^{*}$ and $F^{*}$. And its symbol is the map from a point to $D$ since $T^{*}(x)\cong x$ is one point. 
\\To proceed to the index theorem we need to define the $\bf {elliptic}$ differential operator. 
\begin{definition}
A linear differential operator $D$ is called $\bf{elliptic}$ of order $k$ if, for any $\epsilon_{x}\in T^{*}_{x}\mathcal{X}$ non zero, $$\sigma_{k}(D)(\epsilon_{x}):E_{x}\rightarrow F_{x}$$ is an isomorphism. 
\end{definition}
\begin{definition}
A bounded operator $D:E\rightarrow F$ is called Fredholm if $Ker(D)$ and $Coker(D)$ are both finite dimensional. We denote by $F(E,F)$ the space of all Fredholm operaters from $E$ to $F$. 
\end{definition}
\begin{definition}
The analytical index of a Fredholm operator $D$  is defined to be dim$(\ker D)$-dim(coker D). 
\end{definition}
\theorem An elliptic operator is Fredholm.
\remark This is a general theorem involve functional analysis, we are not able to present the proof in here. The interest reader is invited to read Lecture 2,3,4 of Erik van den Ban and Marius Crainic's notes. 
\theorem For finite dimensional vector bundles $E$, $F$ over a point, the analytical index $|D|$ is equal to dim $E$-dim $F$. 
\begin{proof}  The fact that $D$ is a vector space homomorphism follows from our discussion in Example 1.2.46. The theorem follows from the well-known rank-null identity in linear algebra by assuming dim $E=m$ and dim $F=n$ (over $\mathbb{C}$):
$$\text{dim} ker(f)+\text{dim}im(f)=\text{dim}\mathbb{C}^{m}=m, 
\text{dim} coker(f)+\text{dim}coim(f)=\text{dim}\mathbb{C}^{n}=n$$
Hence 
$\text{dim} ker(f)-\text{dim}coker(f)=\text{dim}\mathbb{C}^{m}-\text{dim}\mathbb{C}^{n}=m-n. $
\end{proof}
\section{The topological index[unfinished]}
We now define the topological index of an operater. 

\chapter{Lie groups and lie algebras[unfinished]}
\remark Much of the material in here are standard. The author first learnt them from \cite{Hall} and is grateful to the author. 
\definition A {\bf{Lie Group}} is a smooth manifold $G$ as well as a group with group operations defined $G\times G\rightarrow G$ such that the map $G\times G\rightarrow G$ sending $(x,y)$ to $x*y^{-1}$ is smooth. 
\example The general linear group $GL(n,\mathbb{R})$ consists of non-singular matrices in $M_{n}(\mathbb{R})$, and the special linear group $SL(n,\mathbb{R}): M\in M_{n},\det(M)=1$. 
\example The orthgonal group $O_{n}:M\in O_{n}\leftrightarrow MM^{t}=1$; and the special orthgonal group $SO(n): M\in O_{n}, \det(M)=1$. 
\example The unitary group $U_{n}: M\in M_{n}(\mathbb{C}), M^{*}M=1$, and the special unitary group $SU_{n}:M\in SU_{n}(\mathbb{C}), \det(M)=1$. 
\example The symplectic group $Sp(n,\mathbb{R}): M\in M_{n}, M^{*}JM=1$ with $J$ defined to be:
$$\begin{bmatrix}
I & 0  \\[6.5pt]
0&   -I 
\end{bmatrix}$$ where $I$ is a $n\times n$ matrix. With field be $\mathbb{C}$ instead of $\mathbb{R}$, we have the symplectic group $Sp(n,\mathbb{C})$. We define the group $Sp(n)$ to be $Sp(n)=Sp(2n,\rr)\cap O(2n)$.
\definition The above groups are called the {\bf{classical groups}}. We shall abbreviate the complex cases as $GL_{n},SL_{n},SO_{n}, SU_{n}, Sp_{n}$. 
\discussion We now give a partial proof of an important lemma. 
\begin{lemma} The classical groups are all Lie groups. 
\begin{proof}
As an outline we need to show that $GL_{n}\mathbb{R}$ is a Lie group ({\bf{i}}). Then we need to show every closed subgroup of a Lie group is a Lie group ({\bf{ii}}).\\ Since the above $GL_{n},SL_{n},U_{n},Sp(n),Sp(n,\mathbb{R}),Sp(n,\mathbb{C})$ are all closed under limit, they are closed subgroups of $GL_{n}\mathbb{R}$(topologically we can always embed a complex manifold in a real manifold). We may view $GL_{n}\mathbb{R}$ as a subset of $\mathbb{R}^{n^{2}}$. As it is defined by the equation $\det(x)\not=0$, while $\det$ is a smooth function,we may conclude that $GL_{n}\mathbb{R}$ is an open subset of $\mathbb{R}^{n^{2}}$, thus admit the structure of a submanifold. On the other hand matrix inverse and multiplication are smooth in $\mathbb{R}^{n^{2}}$ as we may compute the inverse explicitly in linear terms via Kramer's rule. So {\bf{i}} is verified. We still need to verify {\bf{ii}}, which is the {\bf{Cartan's theorem}}. We are not going to prove this in here. The interest reader can consult \cite{Tao}. 
\end{proof}
\end{lemma}
\remark The classical groups are a subset of the more general {\bf{matrix groups}}. A matrix group is any subgroup $H$ of $GL(n;\mathbb{C})$ with the following property: if $A_{n}$ is any sequence in $H$, and $A_{n}$ converges to some matrix $A$, then either $A\in H$, or $A$ is not invertible. It is easy to generalize the above proof to matrix groups. 

\definition The {\bf{Lie algebra}} of a Lie group is defined to be the tangent space at the identity element. \\

\example The Lie algebra of $O_{n}$ is $o_{n}$.

\example The Lie algebra of $SL_{n}$ is $sl_{n}$.

\example The Lie algebra of $GL(V)$ is $gl_{v}$.

Suppose $X$ is a vector field in a Lie group $G$. Since left multiplication is assumed to be smooth, the tangent vector $X_{a}$ represented by the differential of the curve $\gamma_{a}$ at a point $a$ in $G$ is mapped to another tangent vector $X_{ga}$ at the image $ga$ via the differential of the curve $g\gamma_{a}$. Now the action of $g$ on $X$ may be defined as $(a,X_{a})\rightarrow (a,X_{ga})$. 

\definition A {\bf{left invariant vector field }}of a Lie group is defined to be a vector field $X$ such that $X=gX$. 
\\ It is evident that from any element $x$ in the Lie algebra, we may extend to a left-invariant vector field in $G$. And vice versa from a left invariant vector field we may find a vector from the tangent space of the identity element of the Lie group (which we already defined as the Lie algebra). Though this is not rigorous, we somehow reach a one-to-one correspondence between the elements of the Lie algebra and the left invariant vector fields on $G$. To simplify notation we now denote the Lie algebra of $G$ by $\mathcal{G}$. 

\definition An {\bf{integral curve}} $V$ of a vector field $X$ is a parametric curve $f: \mathbb{R}\rightarrow G$ such that $df/dg_{g=g_{0}}=X_{g_{0}}$. It is obvious from the standard existence theorem of ODE (e.g Picard's theorem) that an integral curve start from $g_{0}$ with tangent vector $X_{g_{0}}$ always exists and is unique in the local neighborhood, hence we may speak of the {\bf{maximal integral curve}}.

\definition Given the {\bf{left-invariant}} vector field generated by $X_{g_{0}}\in \mathcal{G}$, for $\lambda\in \R$ the {\bf {exponential map}} $e$ maps $\lambda X_{g_{0}}$ to the image of the maximal integral curve starting with $X_{g_{0}}$ at time $\lambda$. \\ Since left-invariant vector field are well-defined for all $g\in G$, the maximal integral curve in the exponential map actually has domain $\mathbb{R}$. And since the integral curve may be prolonged, we see it is obvious that $$e^{sX_{g_{0}}}e^{tX_{g_{0}}}=e^{(s+t)X_{g_{0}}}$$ Hence the exponential map is a homomorphism. 

We now proceed to the main theorem of this section without a proof. The interested reader should consult \cite{Tao}
\theorem {\bf{Lie's first theorem}} Let G be a Lie group. Then the exponential map is smooth. Furthermore, there is an open neighbourhood $U$ of the origin in $\mathcal{G}$ and an open neighborhood $V$ of the identity in $G$ such that the exponential map $exp$ is a diffeomorphism from $U$ to $V$. 


\section{Lie algebras}
\remark This section is quite standard. It follows \cite{Fulton},\cite{Humphreys},\cite{Sternberg},\cite{Serre} and \cite{Bourbaki}. 
\definition Suppose $G$ is a Lie group, with $\mathcal{G}$ its Lie algebra, the {\bf{inner automorphism}} $\Psi:G\rightarrow Aut(G)$ is given by $\Psi(g)(h)=ghg^{-1}$. 
\definition For a homomorphism $\rho:G\rightarrow H$, the {\bf{adjoint representation}} $Ad(g)=(d\Psi_{g})_{e}:\mathcal{G}\rightarrow \mathcal{G}$ is the differential of the inner automorphism at the identity element. 
\remark The adjoint representation establishes the following commutative diagram:

$$\begin{CD}
T_{e}G @>(d\rho)_{e}>> T_{e}H\\
@VVAd(g)V @VVAd(p(g))V\\
T_{e}G @>(d\rho)_{e}>> T_{e}H
\end{CD}$$
equivalently we have $d\rho (Ad(g)(v))=Ad(\rho(g))(d\rho(v))$. because it is the differential of this diagram at $e$:
$$\begin{CD}
G @>\rho>> H\\
@VV\Psi_{g}V @VV\Psi_{\rho(g)}V\\
G @>\rho>> H
\end{CD}$$
This follows trivially since $$\Psi_{\rho(g)}\rho(h)={\rho(g)}{\rho(h)}(\rho(g))^{-1}=\rho(ghg^{-1})=\rho(\Psi(h))=\rho(ghg^{-1})$$ as $\rho$ is assumed to be a homomorphism. 

\definition The {\bf{Lie bracket}} is the differential of the adjoint representation. In terms of Lie algebra, we have a map $ad:\mathcal{G}\rightarrow End(\mathcal{G})$. In other words we obtained bilinear map:$\mathcal{G}\times \mathcal{G}\rightarrow \mathcal{G}$ as $ad(X)(Y)=Z$. Now we use $[X,Y]=ad(X)(Y)$ for $X,Y\in \mathcal{G}$.\\
\indent It is possible to explicitly compute $[X,Y]$ for classical(in general, matrix) groups. Consider their embedding in $GL_{n}\mathbb{R}$, then $Ad(g)(M)=gMg^{-1}$. Consider a path $\gamma$ with $\gamma(0)=e$ and $\gamma'(0)=X$. Then by our definition of $[X,Y]$ we have $[X,Y]=ad(X)(Y)=\frac{d}{dt}_{t=0}Ad(\gamma(t))(Y)$. Applying the product rule to $Ad(\gamma(t))(Y)=\gamma(t)Y\gamma(t)^{-1}$, we have $$\gamma'(0)Y\gamma(0)+\gamma(0)Y(-\gamma(0)^{-1}\gamma'(0)\gamma(0)^{-1})=XY-YX$$ From now on we shall identify $[X,Y]=XY-YX$ in our project.

\remark For matrix groups because of the commutator relationship we have the reknowned {\bf{Jacobi Identity}} for Lie algebra: $$[X,[Y,Z]]+[Y,[Z,X]]+[Z,[X,Y]]=0$$ This motivates the formal definition independent of Lie group $G$:

\definition A vector space $V$ over field $F$, with an operation $V\times V\rightarrow V$ denoted $(x,y)\rightarrow [x,y]$ is called a Lie algebra if the following axioms are satisfied:
\begin{itemize}
\item The bracket operation is bilinear. 
\item $[x,x]=0,\forall x\in V$. 
\item $[x,[y,z]]+[y,[z,x]]+[z,[x,y]]=0$, $\forall x,y,z\in V$. 
\end{itemize}

\remark The subsequent discussion use $V$ instead of $\mathcal{G}$ to stress that we are dealing with the abstract Lie algebra. 

\definition Let $V,V'$ be Lie algebras. A linear transformation $\phi:V\rightarrow V'$ is called a {\bf{homomorphism}} if $\phi([xy])=[\phi[x]\phi[y]],\forall x,y\in V$. A representation of a Lie algebra $L$ is a homomorphism $\phi:L\rightarrow gl(V)$.

\example The {\bf{adjoint representation}} space (we use the same name in the Lie algebra case) sends $x$ to $ad(x)$. It is clear that $ad$ is a linear transformation since the bracket is bilinear. To see it preserves the bracket, we calculate:
\begin{align}
[ad(x), ad(y)](z)=ad(x)ad(y)(z)-ad(y)ad(x)(z)=[x,[y,z]]-[y,[x,z]]=\\
[x,[y,z]]+[[x,z],y]=[[x,y],z]=ad[x,y](z)
\end{align} Its kernel is obviously $Z(V)$. 

\definition A subspace of a Lie algebra $V$ is called an {\bf{ideal}} of $V$ if $x\in V,y\in I$ implies $[xy]\in I$.

\definition Let $V$ be a Lie algebra, we define $V^{1}=[V,V]$, $V^{2}=[V^{1},V^{1}]$, etc. If $V^{n}$ eventually vanishes, then we say $V$ is {\bf{solvable}}.

\remark Let $V$ be a solvable subalgebra of $gl(V)$, $V$ of finite dimension. Lie proved that the matices of $L$ relative to a suitable basis are upper triangular. This folows from the more general theorem that $V$ must contain some general eigenvector for all the endomorphisms in $L$, which we will not prove in here. 
 
\lemma Let $V$ be a Lie algebra. 
\begin{itemize}
\item If $I$ is a solvable ideal of $L$ such that $V/I$ is solvable, then $V$ itself is solvable. 
\item If $I$, $J$ are solvable ideals of $L$, then so is $I+J$. 
\end{itemize}
\begin{proof}
For the first one, assume $V/I$ is solvable and $(V/I)^{i}=0$, then $V^{i}\in I$. But $I$ is assumed to be solvable, so $V$ must be solvable. The second one follows by establishing the isomorphism between $(I+J)/J$ and $I/(I\cap J)$. The right side is solvable, hence by the first one $I+J$ must be solvable. 
\end{proof}
\definition Let $V$ be a Lie algebra and let $S$ be a maximal solvable ideal. If $I$ is any other maximal solvable ideal of $V$, then by the second part of the lemma above $I+V$ would be an solvable ideal as well. Hence by Zorn's lemma $V$ has a maximal solvable ideal, which we denote as the {\bf{radical}} of $V$({\bf{Rad V}}). In case {\bf{Rad V}}=0, $V$ is called {\bf{semisimple}}. 
\definition Let $v\in V$, $v$ is called {\bf{ad-nilpotent}} if the series $v^{i}=[v,...[v,V]]=0$ in the end. Notice we may write $ad(v)$ in terms of matrices. If $v$ is not ad-nilpotent, then it has a unique decomposition $v=s+n, ad(v)=ad(s)+ad(n)$ where $ad(s)$ consists of diagonal matrices, and $n$ is ad-nilpotent(this follows from the Jordan decomposition). In this case we call $n$ and $s$ as {\bf{semisimple}}.

\indent Assume $V$ is semisimple, then not all elements in $V$ are nilpotent (otherwise $V$ would be equal to its radical). Then we may speak of its subalgebra consisting consisting purely of semisimple elements. We call such a subalgebra {\bf{toral}}.
\lemma A toral subalgebra of $V$ is abelian. 
\begin{proof}
Let $T$ be toral. Since $ad_{T}(x)$ can be diagonalized, we just have to show that $ad_{T}(x)=0$ for all $x\in T$. But this is true only if $ad_{T}x$ has no nonzero eigenvalues. Suppose $[xy]=ay$ for some $y\in T$, then $ad_{T}y(x)=-ay$ is an eigenvector of $ad_{T}y$ of eigenvalue 0. But we can write $x$ as a linear combination of eigenvectors of $ad_{T}y$, then after applying $ad_{T}y$ to $x$ all that is left is a combination of eigenvectors with nonzero eigenvalues. This contradicts the preceding conclusion. 
\end{proof}
\definition Let $V$ be a complex Lie algebra (vector space over $\C$), $H$ its maximal toral subalgebra of $V$. Since $H$ is abelian by the above lemma, $ad_{V}(H)$ is a commuting family of semisimple endomorphisms (see Definition 3.1.16) of $V$. By standard linear algebra $ad_{V}(H)$ is simultaneously diagonalizable. In other words $V$ is the direct sum of subspaces $V_{\alpha}=\{x\in V|[hx]=\alpha(h)x\}$ for all $h\in H$. The set of all nonzero $\alpha\in H^{*}$ for which $V_{\alpha}\not=0$ is denoted by $\Phi$; the elements of $\Phi$ are called the {\bf{roots}} of $V$ relative to $H$. With this notation we call a root space decomposition: $$V=\bigoplus_{\alpha\in \Phi}V_{\alpha}$$ This is called the {\bf{Cartan decomposition}}. 
\remark Notice that when $\alpha=0$ we have $V_{0}=C_{V}(H)$ be the normalizer of $H$. It is not a trivial fact that $H$ is its own normalizer (this result can actually push us to prove $H$ is also the maximal cartan subalgera of $V$). The interest reader should refer to \cite{Humphreys} (p80). In the light of this we have $$V=H\oplus \sum_{\alpha\not=0\in \phi}V_{\alpha}$$
\discussion In the light of the correspondence between Lie algebra and Lie group, we may prove that all the maximal toral subalgebras must be conjugates of one another, and any element of $\mathcal{G}$ belongs to one such maximal toral subalgebra. Further, every such maximal toral subalgebra corresponds to a maximal torus in the Lie group. The interested reader is invited to read \cite{Varadarajan}, page 367. 
\definition Consider the adjoint representation of the Lie algebra on itself, the {\bf{Killing form}} of a Lie algebra $V$ is a bilinear form defined by $$k(x,y)=\text{Tr}[ad(x)ad(y)]$$
\lemma Suppose $x,y,z$ are endomorphisms of $V$. If $V$ is finite dimensional, then $\text{Tr}([x,y]z)=\text{Tr}(x[y,z])$. 
\begin{proof}
This follows since $[x,y]z=xyz-yxz$, and $x[y,z]=xyz-xzy$, with $\text{Tr}(y(xz))=\text((xz)y)$.
\end{proof}
\lemma The Killing form is associative: $k([x,y],z)=k(x,[y,z])$. 
\proof 
This follows from the identity $$\text{Tr}([x,y]z)=\text{Tr}(x[y,z])$$ we proved above and the derivation relation $ad([x,y])=[ad(x),ad(y)]$ in the example above. 
\remark We remark without proof that if $V$ is semisimple, the Killing form $k$ is nondegenerate on $V$. In other words there does not exist element $v\in V$ such that $k(v,V)=0$. Use this result and the fact $H=C_{V}H$ noted earlier, we may prove that the restriction of $k$ to $H$ is nondegenerate as well. This allows us to associate to every linear function $\phi$ a unique element $t_{\phi}\in H$ given by $$\phi(h)=k(t_{\phi},h)$$ 
\lemma From {\cite{Sternberg}} We now claim the following relationship:
\begin{itemize}
\item $\Phi$ spans $H^{*}$. 
\begin{proof} Otherwise there exist $h$ such that $\alpha(h)=0$ for all $\alpha\in \Phi$. This implies $[h,V_{\alpha}]=0$ for all $\alpha$. So $[h,V]=0$. But this implies $h\in Z(V)$. Since $V$ is assumed to be semisimple, $h=0$. This contradicts with the hypothesis. 
\end{proof}
\item If $\alpha,\beta\in H^{*}$, and $\alpha+\beta\not=0$, then $V_{\alpha}$ is orthogonal to $V_{\beta}$ under the Killing form. 
\begin{proof} In other words, $\alpha$ and $\beta$ are orthogonal as well. Assume otherwise, then let $x\in V_{\alpha},y\in V_{\beta}$, associativity of the Killing form allows us to write $k([hx],y)=-k([xh],y)=-k(x,[hy])$. But this implies $\alpha(h)k(x,y)=-\beta(h)k(x,y)$, in other words $$(\alpha+\beta)k(x,y)=0$$ This force $k(x,y)=0$, hence $x$ is orthogonal to $y$, and $\alpha$ orthgonal to $\beta$. 
\end{proof}
\item If $\alpha\in \Phi$, then $-\alpha\in \Phi$ as well. 
\begin{proof}
Otherwise $V_{\alpha}\bot V$.
\end{proof}
\item $x\in V_{\alpha}$, $y\in V_{-\alpha}$, $a\in \Phi$ implies $[x,y]=k(x,y)t_{\alpha}$. 
\begin{proof}
We have \[k(h,[x,y])=k([h,x],y)=k(t_{a},h)k(x,y)=k(k(x,y)t_{a},h)=k(h,k(x,y)t_{a})\]
\end{proof}
\item 
$[V_{\alpha},V_{-\alpha}]$ is one dimensional with basis $t_{\alpha}$. 
\begin{proof}
This follows from the preceding and the fact that $V_{\alpha}$ cannot be perpendicular to $V_{-\alpha}$, as other wise it would be orthogonal to all of $V$.   
\end{proof}
\item $\alpha(t_{\alpha})=k(t_{\alpha},t_{\alpha})\not=0$. 
\begin{proof}
Otherwise we may choose $x\in V_{\alpha}$, $y\in V_{\beta}$ with $k(x,y)=1$, we get $$[x,y]=t_{a},[t_{a},x]=[t_{a},y]=0$$ So $x,y,t_{a}$ span a solvable three dimensional algebra. By Lie's theorem cited earlier, its adjoint action must be upper triangular. Hence $ad t_{a}$ must be nilpotent as it is in the commutator algebra of this subalgebra. But we know it is semisimple since it lies in $H$, which is the maximal toral subalgebra. This says $ad_{V}t_{\alpha}=0$. In other words and lie in the center, but this is impossible for we assumed $V$ is semisimple. 
\end{proof}
\item For every $\alpha$ there are $e_{\alpha}\in V_{\alpha}$, $f_{\alpha}\in V_{-\alpha}$, and $h_{\alpha}\in H$ such that $e_{\alpha},f_{\alpha},h_{\alpha}$ is isomorphic to $sl(2)$. 
\begin{proof}
We choose $e_{\alpha}$ and $f_{\alpha}$ satisfying $k(e_{\alpha},f_{\alpha})=\frac{2}{k(t_{a},t_{a})}$. Now if we set $h_{a}=\frac{2}{k(t_{a},t_{a})}$, then $e_{\alpha}$, $f_{\alpha}$, $h_{\alpha}$ span a subalgebra isomorphic to $sl(2)$, which is generated by 
$$\displaystyle e_{\alpha}\rightarrow
\left(\begin{array}{ccc}
0 & 1 \\
0 & 0 \\
\end{array}\right), f_{\alpha}\rightarrow
\left(\begin{array}{ccc}
0 & 0 \\
1 & 0 \\
\end{array}\right), h_{\alpha}\rightarrow
\left(\begin{array}{ccc}
1 & 0 \\
0 & -1 \\
\end{array}\right)$$
This is evident since we have $[e_{\alpha},f_{\alpha}]=h_{\alpha}$. $[h_{\alpha},e_{\alpha}]=\frac{2}{\alpha(t_{\alpha})}[t_{\alpha}x_{\alpha}]=\frac{2\alpha(t_{\alpha})}{\alpha(t_{\alpha})}e_{\alpha}=2e_{\alpha}$, similarly $[h_{\alpha},f_{\alpha}]=-2f_{\alpha}$. 
\end{proof}
\end{itemize}
\discussion Let $V$ be a semisimple Lie algebra. It has a {\bf{Cartan decomposition}}. For any $\alpha\in \Phi$, by the above discussion we may pick $h_{\alpha}\in [V_{\alpha},V_{-\alpha}]$ and $\alpha(h_{\alpha})=2$. Now elements of $\Phi$ span a vector space $E$. For any $\alpha$ we introduce the reflection $W_{\alpha}$ that reflects $E$ by the hyperplane $\Omega_{\alpha}=\beta\in \Phi:\langle h_{a},\beta \rangle=0$. In other words, $$W_{\alpha}(\beta)=\beta-\frac{2\beta(h_{\alpha})}{\alpha(h_{\alpha})}\alpha=\beta-\beta(h_{\alpha})\alpha$$ 

\definition The group generated by $W_{\alpha}$ is called the {\bf{Weyl group}}. 
\definition In the light of the connection between Lie groups and Lie algebras ({\bf{Lie's theorem}} in the Lie group section), we may define {\bf{Weyl group}} via the Lie group. Let $G$ be a Lie group and $\mathcal{G}$ its Lie algebra, then we may define the {\bf{Weyl group}} (denoted by W) by $$N(T)/T$$ In here $T$ is the maximal torus in the Lie group. The author could {\bf{not}} find from standard reference that the two definitions indeed coincide(though they are).

\section{Representation theory}
This section follows from \cite{Fulton} as well as \cite{Serre3}. Our discussion in here are restricted to finite groups. 
\definition A {\bf{representation}} of a group $G$ on $V$ is a homomorphism $\phi: G\rightarrow GL(V)$. A {\bf{subrepresentation}} of a group $G$ on $V$ is a vector subspace $W$ that admits $\phi|_{W}$ as a homomorphism. The {\bf{degree}} of the representation is the dimension of $V$. 

\definition A left $R$-{\bf{module}} is an abelian group $A$ with a ring $R$ that admits a group homormophism $f:R\rightarrow End(A)$ which is associative. In other words we have:
\begin{itemize}
\item $r(a+b)=ra+rb$.
\item $(r+s)a=ra+sa$.
\item $r(sa)=(rs)a$.
\end{itemize}

\remark We say that such a map gives $V$ the structure of a {\bf{$G$-module}}, for we can identity $A$ with $V$ and $G$ with its image in $GL(V)$. It is possible to define a more general {\bf{$F[G]$-module}} by the group ring $FG$, whose elements are of the form $\sum f_{i}g_{i}, f_{i}\in F,g_{i}\in G$, and multiplication is defined by $(f_{i}g_{i})(f_{j}g_{j})=(f_{i}f_{j})g_{i}g_{j}$. More generally we call $F[G]$ as the {\bf{group algebra}}. 

\definition A {\bf{$G$-linear map}} between two representations $V$ and $W$ of $G$ is a vector space map $\phi:V\rightarrow W$ such that $g(\phi(v))=\phi(g(v)),\forall v\in V$. If $\phi$ has an inverse $\phi^{-1}$ such that $\phi*\phi^{-1}=1_{V}$, then we say $\phi$ is an isomorphism. 
$$\begin{CD}
V @>\phi>> W\\
@VVgV @VVgV\\
V @>\phi>> W
\end{CD}$$
\remark Since representation can be treated as $G$-module, the direct sum  of $V$ and $W$ is also a representation. For the tensor product we cannot define $rv\otimes w=v\otimes rw$ as usual case for $R$-modules, which will not make sense. So we define the tensor product as a representation of $G$ via $g(v\otimes w)=gv\otimes gw$. Similarly $Hom(V,W)$ may be view as a representation of $G$ via $(g\phi)(v)=g\phi(g^{-1}v)$, which makes the following diagram commutative:
$$\begin{CD}
V @>\phi>> W\\
@VVgV @VVgV\\
V @>g\phi>> W
\end{CD}$$
\definition We define the subspace $Hom_{G}(V,W)$ as the subspace of $Hom(V,W)$ fixed under the action of $G$. It is obvious that this is just the space of all $G$-linear maps. 
\example The {\bf{permutation representation}}: Assume $G$ has certain group action on $X$. Now let $V$ be the vector space with basis $e_{x},x\in X$, and let $G$ act on $V$ by $g*\sum a_{x}e_{x}=\sum a_{x}e_{gx}$. 
\example The {\bf{regular representation}}: Let $V$ be the vector space spanned with characteristic function $e_{g}$, which takes value 1 on $g$, $0$ on other elements of $G$. Then the permutation representation on $V$ is called the regular represenation. 
\remark By a representation of $F[G]$ on a vector space $V$ we mean an algebra homomorphism $$FG\rightarrow \text{End}(V)$$ so that a representation $V$ of $FG$ means a left $FG$-module. Now a representation $\rho:G\rightarrow \text{Aut}(V)$ will extend by linearity to a map $\overline{\rho}:FG\rightarrow V$. Hence the representations of $FG$ correspond exactly to representations of $G$. The left $FG$-module that acts on $FG$ itself corresponds to the regular representation.\\
The discussion here is to try to motivate the reader on the similarity of the $\C G$ module decomposition and the Peter-Weyl theorem we will discuss later. 
\definition A representation that can be expressed as a direct sum of others is called a {\bf{reducible}} representation. The ones cannot are called {\bf{irreducible}} representation. In terms of $G$-modules, a $G$-module $V$ is called {\bf{irreducible}} if it has precisely two $G$-submodules(0 and itself). 
\lemma{Schur's lemma} \\If $V$ and $W$ are irreducible representations of $G$ and $\phi:V\rightarrow W$ is a $G$-module homomorphism, then
\begin{itemize}
\item Either $\phi$ is an isomorphism, or $\phi=0$.
\item If $V=W$, then $\phi=\lambda I$ for some $\lambda\in \mathbb{C}$, where $I$ is the identity. 
\end{itemize}
\begin{proof}
The first claim follows from the fact that $\ker \phi$ and Im$\phi$ are invariant subspaces (hence submodules). For the second, since $\mathbb{C}$ is algebraically closed, $\phi$ must have an eigenvalue $\lambda$. By the first claim $\phi-\lambda I=0$, hence $\phi=\lambda I$. 
\end{proof}
\discussion By Schur's lemma, we have $\dim Hom_{\C}(C,\bigoplus W_{i})=\sum \dim Hom_{\C}(V,W_{i})$. If $V$ is irreducible then $\dim \text{Hom}_{G}(V,W)$ is the mutiplicity of $V$ in $W$. Similarly if $W$ is irreducible then it is the mutiplicity of $W$ in $V$. In the case where both $V$ and $W$ are irreducible, we have:
\begin{center}
If $V\cong W$, then $\dim \text{Hom}_{G}(V,W)=1$; If $V\not\cong W$, then $\dim \text{Hom}_{G}(V,W)=0$. 
\end{center}
This is usually called {\bf{the intersection number}} of two representations. It can be regarded as a bilinear form. 
\definition If $V$ is a representation of $G$, its {\bf{character}} $\chi_{v}:G\rightarrow \C$ is the complex valued function on the $G$ defined by $$\chi_{V}(g)=\text{Tr}(g|_{V})$$ the trace of $g$ on $V$. In particular, we have $\chi_{V}(hgh^{-1})=\chi_{V}(g)$, so $\chi_{V}$ is constant on the conjugacy classes of $G$. Such a function is called a {\bf{class function}}. 
\remark Let $V$ and $W$ be representations of $G$. Then $$\chi_{V\oplus W}=\chi_{V}+\chi_{W}$$ $$\chi_{V\otimes W}=\chi_{V}\chi_{W}$$ $$\chi_{V^{*}}=\overline{\chi}_{V}$$ The proof is obvious. 
\lemma Consider $V_{G}$ as the subspace of $V$ invariant under $G$. If $G$ is finite, the dimension of $V_{G}:g(v)=v,\forall g\in G$ is $$\frac{1}{|G|}\sum_{g\in G}\chi_{V}(g)$$
\begin{proof}
We notice that $\phi(v)=\frac{1}{|G|}\sum_{g\in G}(g)(v)$ is the projection map from $V$ to $V_{g}$. Taking its trace we found its dimension. 
\end{proof}
\remark It is important to remark that from the discussion, because $V,W$ are finite dimensional, since $\text{Hom}(V,W)\cong V^{*}\otimes W$, we have $$\chi_{\text{Hom}(V,W)}(g)=\overline{\chi_{V}(g)}\cdot \chi_{W}(g)$$ which implies 
\begin{center}
$\frac{1}{|G|}\sum_{g\in G}\overline{\chi_{V}(g)}*\chi_{W}(g)=0$ if $V\not=W$\\ $\frac{1}{|G}\sum_{g\in G}\overline{\chi_{V}(g)}*\chi_{W}(g)=1$ if $V=W$. 
\end{center}
In other words, if we define inner product on the representation ring of $G$ via $$(\alpha,\beta)=\frac{1}{|G|}\sum_{g\in G}\overline{\alpha(g)}\beta(g)$$ the characters of the irreducible representations of $G$ are orthonormal. Further, a representation $V$ is irreducible if and only if $(\chi_{V},\chi_{V})=1$. Assume all $W_{j}$ distinct, if $V\cong \sum W_{i}^{k_{i}}$, then $(\chi_{V},\chi_{V})=\sum k_{i}^{2}$.
\example Let $G$ be any finite group, $R$ be the regular representation of $G$, then the character of $R$ is
\begin{center}
$\chi_{R}(g)=0$, if $g\not=e$\\ $\chi_{R}(g)=|G|$, if $g=e$. 
\end{center}
Thus $R$ is not irreducible if $G\not=\{e\}$. If we assume $R=W_{i}^{k_{i}}$, then $$k_{i}=(\chi_{W_{i}},\chi_{R})=\frac{1}{|G|}\chi_{W_{i}}(e)*|G|=\dim W_{i}$$ Thus any irreducible representation $V$ of $G$ appears in the regular representation $\dim W_{i}$ times. This showed, for example, there are only finitely many irreducible representations if $G$ is finite. 
\lemma We have the formula $$|G|=\dim R=\sum_{i}\dim(V_{i})^{2}$$
\begin{proof}
This is obvious from the above example. In particular, if we apply the above formula to the value of the character of the regular representation of an element other than the identity, we have $0=\sum(\dim W_{i})\cdot \chi_{V_{i}}(g)$, if $g\not=e$.
\end{proof}
\definition $G$ is called {\bf{complete reducible}} if any representation of $G$ is a direct sum of irreducible representations. 

\begin{lemma} Assume $F$ has characteristic 0. If $W$ is a subrepresentation of a representation $V$ of a {\bf{finite group}} $G$, then there is a complementary invariant subspace $W'$ of $V$, so that $V=W\oplus W'$. 
\end{lemma}

\begin{proof}
We choose an arbitrary subspace $U$ complementary to $W$. Let $\pi_{0}$ be the projection given by the direct sum decomposition $V=W\oplus U$. We average the map over $G$ by defining $$\pi(v)=\sum_{g\in G}g(\pi_{0}(g^{-1}(v))).$$
This will be a $G$-linear map from $V$ onto $W$, since $\pi(v)=g\pi(g^{-1}v)$. On $W$ it is simply multiplication by $|G|$(since $char F=0$). Its kernel will therefore be a subspace of $V$ invariant under $G$ and complement to $W$.  
\end{proof}
\remark Thus it folows finite groups over a vector space of characteristic 0 is completely reducible. And every representation of $G$ is of the form $\oplus W_{i}^{k_{i}}$ where $W_{i}$ are $G$'s irreducible representations. \\
We prove these theorems for finite groups to motivate the similar situation in Lie groups where Frobenius reciprocity would play a major role. We shall discuss that further in section about Bott's paper.  
\section{Interlude}
Weyl proved the following theorem for semisimple Lie algebra. This enable us, for example, to speak of the {\bf{representation ring}} of $G$, which we will define later. 
\theorem {\bf{Weyl's theorem}}: 
\begin{center}
If $\mathcal{G}$ is a semisimple Lie algebra, then $\mathcal{G}$ is completely reducible. 
\end{center}
\example Let $G$ be $S_{n}$ operating on $\mathbb{C}^{n}\cong \oplus_{i=1}^{n} \mathbb{C}e_{i}$.The {\bf{standard representation}} of $G$ on the subspace $\sum a_{i}=0$ is given by permuting the basis elements. 
\example Let $G$ be a matrix group. By definition it can be embedded in $GL(n,\mathbb{C})$. The {\bf{standard representation}} of $G$ is given by its action on the space $\mathbb{C}^{n}$. For example, $SU(2)$ has a standard representation of degree 2. 
\discussion In the following we discussion the $\mathbb{C}[G]$ module case, where $G$ is assumed to be finite. Although this is irrevlant to our project, it serves a good motivation for the analgous {\bf{Peter-Weyl}} theorem which will be integral to Bott's theorem in later section. 

\theorem Let $\{W_{i}\}$ denote the irreducible representations of $G$. Then $\mathbb{C}[G]\cong \bigoplus \text{End}(W_{i})$ 
\begin{proof}
For any representation $W$ of $G$, the map $G\rightarrow \text{Aut}(W)$ extends by linearity to a map $\mathbb{C}[G]\rightarrow \text{End}W$. Now applying this to each of the irreducible representations $W_{i}$ gives us a canoical map $$\phi:\mathbb{C}[G]\rightarrow \bigoplus End(W_{i})$$ We claim $\phi$ is an isomorphism: It is obviously injective. By the lemma above both sides have dimension $\sum \dim W_{i}^{2}$. Hence $\phi$ is an isomorphism. 
\end{proof}
\section{Lie group, Lie algebra revisited}
\definition A {\bf{representation}} of a Lie algebra $\mathcal{G}$ on a vector space $V$ is simply a map of Lie algebras $$\rho: g\rightarrow gl(V)=\text{End}(V)$$ a linear map that preserves brackets, such that $$[X,Y](v)=X(Y(v))-Y(X(v))$$ 
\remark In the light of {\bf{Lie's first theorem}} which establishes the correspondence between Lie group and its corresponding Lie algebra, we may claim the following:(quoted verbatim from \cite{Fulton})
\theorem
If $G$ and $H$ are Lie groups with $G$ connected and simply connected, the maps from $G$ to $H$ are in one-to-one correspondence with maps of the associated Lie algebras, by associating $\rho:G\rightarrow H$ its differential $(d\rho)_{e}:g\rightarrow h$. 
\remark 
This implies, in particular that the representations of a connected and simply conneced Lie group are in one-to-one correspondence with representations of its Lie algebra. It also has some unexpected consequence, for example Lie algebra act on the tensor product by derivation: $X(v\otimes w)=X(v)\otimes w+v\otimes X(w)$. Hence suppose $v,w$ are eigenvectors for $\mathcal{G}$ with eigenvalue $\alpha$ and $\beta$, then $$X(v\otimes w)=\alpha(v\otimes w)+\beta(v\otimes w)=(\alpha+\beta)v\otimes w$$ 

\definition A Lie group is called {\bf{semisimple}} if its Lie algebra is semisimple. 
\definition Let $G$ be a semisimple Lie group and $\mathcal{G}$ its associated Lie algebra. Assume the Cartan decomposition of $\mathcal{G}$ by adjoint representation is $$\mathcal{G}\cong H\bigoplus_{\alpha\in \Phi\not=0} \mathcal{G}_{\alpha}$$ where $\Phi$ is the set of all roots. Since $H$ is abelian, its action can be diagonalized on $V$. In other words $V$ admit a direct sum decomposition \[
V=\bigoplus V_{\alpha}
\] where $\alpha\in H^{*}$. So $h(v)=\alpha(h)(v)$ for all $h$ in $H$,$v$ in $V$. $\alpha$ is called a {\bf{weight}} of the representation $V$, and the corresponding $V_{\alpha}$ are called {\bf{weight spaces}}. 

Let $\beta$ be a root of the adjoint representation in $H^{*}$, $\alpha$ be a weight of $V$. Let $g_{\beta}\in \mathcal{G}_{\beta}$, $v_{\alpha}\in V_{\alpha}$. We claim that $$g_{\beta}(v_{\alpha})\in V_{\alpha+\beta}.$$ In other words the action of root spaces 'shift' the weight space by the root vector. We have the relationship $$h(g_{\beta}v_{\alpha})=g_{\beta}h(v_{\alpha})+[h,g_{\beta}]v_{\alpha}=\alpha(h)g_{\beta}(v_{\alpha})+\beta(h)g_{\beta}v_{\alpha}=(\beta+\alpha)(h)g_{\beta}(v_{\alpha})$$ Hence $g_{\beta}v_{\alpha}\in V_{\beta+\alpha}$. If we assume $V$ to be finite dimensional, then it is obvious there must be a maximal vector among $V_{\beta+n\alpha}$. In general such a weight vector is called {\bf {the highest weight vector}} or {\bf {dominant weight vector}}. 
\remark Assume $V$ to be irreducible, it is not a trivial fact that $V$ is a {\bf{cyclic module}} generated by its dominant weight vector, for reference, see \cite{Sternberg} Chapter 7. To explain what is a cyclic module we may envision the shift along a given vector by all other vectors. 
\\\indent In other words, starting with a maximal weight vector $v$, we may apply elements in $\mathcal{G}_{\alpha}$ to shift $v$ until it is invariant under the Weyl group, then we (recover) $V$ from the maximal weight vector. Thus the problem of finding irreducible representations of $G$ translates to find maximal weight vectors of $V$.  

\section{Bott's theorem[unfinished]}
Let $\phi:\mathcal{G}\rightarrow gl(V)$ be a (finite dimensional) representation of a semisimple Lie algebra. Then $\phi$ is completely reducible. 
\definition The {\bf{representation ring}} of $G$ is generated by its irreducible representations with two operartions: $+=\oplus$ and $\times=\otimes$. To define the difference for two representations, we introduce the formal difference: $E-F=G-H\leftrightarrow E\oplus H\cong G\oplus F$. It is easy to verify that all the axioms for a commuative ring is satisfied: 
This ring is denoted by $R(G)$. Thus, for $G$ that is completely reducible, the elements of $R(G)$ are of the form $\sum a_{\lambda}V_{\lambda}$, where $V_{\lambda}$ ranges over the irreduicble representations of $G$. 

\remark The elements of the representation ring is called {\bf {virtual representations}} of $G$. We may take the {\bf{formal completion}} of $R(G)$, whose elements consists of infinite sums $\sum a_{\lambda}V_{\lambda}$.

\chapter{My work on $SU(n)$}
Throughout this section and follows, capitalized letters imply the Lie group, whereas lower case letters imply the Lie algebra. 
\section{General introduction}
This section follows Chapter 23 of \cite{Fulton}. \\
For semisimple Lie algebras in general, since they are completely reducible we may describe them in terms of irreducible subrepresentations. Let $\Lambda$ be the weight lattice of $\mathcal{G}$, and let $\mathbb{Z}[\Lambda]$ be the integral group ring on the abelian group $\Lambda$. We now define $e_{\lambda}$ for the basis element correpsonding to the weight $\lambda$. The elements of $\mathbb{Z}[\Lambda]$ are expressions of the form $\sum n_{\lambda}e_{\lambda}$. Elements of this ring have all but finitely many number of terms being zero. 

We now define a {\bf{character homomorphism}} as
Char: $R(\mathcal{G})\rightarrow \mathbb{Z}[\Lambda]$
by the formula $\text{Char}[V]=\sum \dim(V_{\lambda})e(\lambda)$. In here $V_{\lambda}$ is the weight space of $V$ for the weight $\lambda$ and $\dim V_{\lambda}$ is its multiplicity. This is clearly an additive homomorphism. We claim the follows:
\begin{itemize}
\item The map is injective. This is because a representation of $\mathcal{G}$ is determined by its weight spaces. Since each irreducible subrepresentation is a cyclic module generated by the highest weight vector, this follows naturally. 
\item The product in the group ring $\mathbb{Z}[\Lambda]$ is determined by $e(\alpha)\cdot e(\beta)=e(\alpha+\beta)$. Under this definition we claim that Char is a ring homomorphism. Clearly it is an additive homomorphism. Now we notice the fact that $$\displaystyle(V\otimes W)_{\lambda}\cong_{\mu+v=\lambda}V_{\mu}\otimes W_{v}$$ 
\item The image of Char is contained in the ring of invariants $\mathbb{Z}[\Lambda]^{W}$. This comes down to the fact that, for any irreducible representation $V$, the weight spaces is invariant under the action of the Weyl group. 
\end{itemize}

It is interesting to investigate the precise relationship between representation ring and character map. To do that we need to introduce the notion of {\bf fundamental weights}. 

\definition Let $\Phi$ be the union of roots. A subset $\Delta$ of $\Phi$ is called a {\bf{base}} if:
\begin{itemize}
\item Let $E$ be the vector space generated by $\Phi$, then $\Delta$ is a basis of $E$.
\item Each root $\beta$ can be written as $\beta=\sum k_{\alpha}\alpha,(a\in \Delta)$, with integral coefficients $k_{\alpha}$ all nonnegative or all nonpositive. 
\end{itemize}
It is not trivial that every root system contains a base (proving that would require a theorem establishing the correspondence between bases and Cartan subalgebras), so we omit that here. 
\definition A \{Weyl chamber\} is one of the maximal connected components in the weight space that is moved by the Weyl group.

Now that we have a definition of base, let $\Delta$ be $\{\alpha_{1},...\alpha_{l}\}$. The vectors $\frac{2\alpha_{i}}{(\alpha_{i},\alpha_{i})}$ again form a basis for $E$. Now $\lambda_{i}$ be the due basis such that $\frac{(2\lambda_{i},\alpha_{j})}{(\alpha_{j},\alpha_{j})}=\Delta_{ij}$. Hence all $(\lambda_{i},\alpha)$ are nonnegative integers, and any $\lambda$ such that $(\lambda,\alpha)\in \mathbb{Z}$ for all $\alpha\in \Phi$ must be generated by $\lambda_{i}$. We call $\lambda_{i}$ the {\bf{fundamental weights}}. Geometrically they are the weights along the edges of the Weyl chamber. Let $\Gamma_{i}$ be the classes in $R(\mathcal{G})$ of the irreducible representations with highest weights $\lambda_{i}$. 

\theorem We claim that $\mathbb{Z}[\Lambda]^{W}$ is a polynomial ring on the variables Char$(\Gamma_{i})$. If we take variables $U_{i}$ and map the polynomial ring on the $U_{i}$ to $R(\mathcal{G})$ by sending $U_{i}$ to $\Gamma_{i}$ we have $$\mathbb{Z}[U_{1},...U_{n}]\rightarrow R(\mathcal{G})\rightarrow \mathbb{Z}[\Lambda]^{W}$$ which means:
\begin{itemize}
\item The representation ring $\mathcal{G}$ is a polynomial ring on the variables $\Gamma_{i}$.
\item The homomorphism $R(\mathcal{G})\rightarrow \mathbb{Z}[\Lambda]^{W}$ is an isomorphism.
\end{itemize}

\remark For the proof we refer the reader to \cite[p.376]{Fulton}. A closely related theorem can be found in \cite{Bott2}, that if $G$ is compact and simply-connected, then $R(G)$(complex representations) is a polynomial ring. 

\example Now in particular for classical groups, its representation ring is isomorphic to $\mathbb{Z}[\Lambda]^{W}$. In light of this we may identify $R(sl_{3})=\Z[x+y+z,xy+yz+zx]/(xyz-1), R(T)=\Z[x,y,z]/(xyz-1)$, for example. In here $x,y,z$ are the weights $L_{i}$ such that $L_{i}(h_{j})=\delta_{ij}$ where $h_{i}$ is the $i$th element in $H$ that has 1 on $i$ position and 0 elsewhere. 

\section{Steinberg's basis}
We briefly summarize the principal theorem Steinberg proved in his paper \cite{Steinberg}. Sterinberg give a basis of $R(G)$ over $R(T)$ on simply-connected Lie groups as follows:

\discussion Let $T$ be the maximal torus of $G$, and $X$ the character group of $T$. Then we know $R(G)\cong \Z[X]^{W}$ via restriction to $T$, even if $G$ is not semisimple. Thus $R(T)$ is a free-module over $R(G)$ is equivalent to $\Z[X]$ be a free-module over $\Z[X]^{W}$. 

\definition Let $\Phi$ denote the root system of $G$ relative to $T$, $\Phi^{+}$ the set of positive roots and $\mathcal{A}$ the corresponding basis of simple roots. Then since $sl_{n}$ is simply-connected, we claim the fundamental weights $\{\lambda_{\alpha}\}$ defined by $(\lambda_{\alpha},b^{*})=\delta_{\alpha b}$ with $b^{*}=\frac{2b}{(b,b)}$ being the coroot of $b$. For any $v\in W$, let $\lambda_{v}$ denote the product in $X$ of those $\lambda_{\alpha}$ for which $\alpha\in \mathcal{A}$ and $v^{-1}a<0$.

\begin{theorem}[Steinberg]: In the principal case, we get $\Z[X]$ over $\Z[X]^{W}$ with $\{w^{-1}\lambda_{w}|w\in W\}$ as a basis. 
\end{theorem}
\begin{discussion} Steinberg proved this theorem using pure classical Lie algebra methods. But despite its clearity, the above formula does not give us an explicit formula for the basis elements regarding $R(T)$ in terms of $R(sl_{n})$. Thus we wish to do explicit computation in this case. 
\end{discussion}
\example 
We now shall carry out the computation for $sl(n,\C)\cong su(3)$. In general the cartan subalgebra $h$ is generated by elements $\mathcal{H}=\{A_{i}:e_{ii}-e_{i+1,i+1}\}$ of dimension $n-1$. Thus the basis for roots can be chosen to be $A_{i}^{*}$ as $e_{i,i}^{*}-e_{i+1,i+1}^{*}$. If we select $e_{i,i}^{*}-e_{j,j}^{*},\{i<j\}$ as the positive roots, and $\mathcal{A}=e_{i,i}^{*}-e_{i+1,i+1}^{*},i\in\{1,..n-1\}$ the simple roots, then the correspdoning well-known fundamental weights are $w_{i}=\sum_{k=1}^{i} e_{k,k}$. Thus Steinberg's formula can be explicitly computed. 

For the $n=3$ case. And $\mathcal{A}$ consists of $A_{1}=e_{1,1}^{*}-e_{2,2}^{*}$ and $A_{2}^{*}=e_{2,2}^{*}-e_{3,3}^{*}$ respectively. The fundamental weights are $\lambda_{A_{1}}=e_{1,1}^{*}$ and $\lambda_{A_{2}}=e_{1,1}^{*}+e_{2,2}^{*}$. The Weyl group $W\cong S^{3}$ acts on permuting the indices, thus we have its six elements to be $1,(12)3,(13)2,(23)1,(123),(132)$. 

For $(1)$ we have $\lambda_{1}$ as a basis element. But we {\bf{do not}} have any element in $\mathcal{A}$ such that $a<0$ since both $A_{1}^{*}$ and $A_{2}^{*}$ are positive. Therefore $1$ can only correspond to the identity element 0. 

For $(12)$ we have $(21)\lambda_{12}$ as a basis element. By definition we have $\lambda_{12}=\prod \lambda_{a}$ such that $(21)(a)<0$. The only element satisfying the role is $A_{1}$, and $\lambda_{A_{1}}=e_{1,1}^{*}$. Thus we have the basis element to be $(21)e_{1,1}^{*}=e_{2,2}^{*}$. 

For $(13)$ we have $(31)\lambda_{13}$, where as $\lambda_{13}=\prod \lambda_{a}:(31)(a)<0$. Both elements satisfy the relationship and we have $\lambda_{A_{1}}^{*}\otimes \lambda_{A_{2}}^{*}$ having highest weight $2e_{1,1}^{*}+e_{2,2}^{*}$. Thus the basis element turns out to be $(31)(2e_{1,1}^{*}+e_{2,2}^{*})=2e_{3,3}^{*}+e_{2,2}^{*}$. 

For $(23)$ we have $(32)\lambda_{23}$, and $\lambda_{A_{2}}=e_{1,1}^{*}+e_{2,2}^{*}$. Thus $(32)e_{1,1}^{*}+e_{2,2}^{*}=[e_{1,1}^{*}+e_{3,3}^{*}]$.

For $(132)$, its inverse is $(123)$. We have $(123)\lambda_{132}=(123)\prod \lambda_{a}:(123)(a)<0$. We have the basis to be $(123)(e_{1,1}^{*}+e_{2,2}^{*})=e_{2,2}^{*}+e_{3,3}^{*}$. 

For $(123)$, its inverse is $(132)$. We have $(132)\lambda_{123}=(132)\prod \lambda_{a}:(132)(a)<0$. This corresponds to $(132)(e_{1,1}^{*})=e_{3,3}^{*}$.

In vector forms, we found the basis for $SU(3)$ is $$(0,0,0),(0,1,0),(0,1,2),(1,0,1),(0,1,1),(0,0,1)$$ In other words it is $$\{1, y,z, xz, yz,  yz^{2}\}$$

\remark We omit a similar discussion on $SU(4)$ because 24 elements would take too much space. My advisor Gregory Landweber has calculated this on the computer. We remark that Steinberg's basis can be viewed geometrically as the lowest weights such that when perturbed we get one weight in each chamber. In general Steinberg's basis is difficult to calculate. 
\section{$n=2$}
We shall work with Lie algebra $su_{n}$ since $SU_{n}$ is simply-connected. For all purposes we identify $su_{n}$ and $sl_{n}\mathbb{C}$ because the complexification does not alter the lie algebra structure. 
\discussion This is the most basic and important case. $SU(2)$ and $SL(2,\C)$ are the Lie groups corresponding to $sl_{2}\mathbb{C}$. The subsequent discussion is a brief summary of Fulton and Harris, Chap 11.1, Chap 23.2. As Lie algebra, $sl_{2}\mathbb{C}$ is generated by 
$$\displaystyle X\rightarrow
\left(\begin{array}{ccc}
0 & 1 \\
0 & 0 \\
\end{array}\right), Y\rightarrow
\left(\begin{array}{ccc}
0 & 0 \\
1 & 0 \\
\end{array}\right), H\rightarrow
\left(\begin{array}{ccc}
1 & 0 \\
0 & -1 \\
\end{array}\right)$$
satisfying $$[H,X]=2X,[H,Y]=-2Y,[X,Y]=H$$  Hence by our discussion in the representation theory earlier, $X\in \mathcal{G}_{2}$, $Y\in \mathcal{G}_{-2}$. Now by the weight space decomposition, let $V$ be an irreducible representation for $sl_{2}\mathbb{C}$. Then its weight spaces $V_{\alpha}$ must be invariant under the action of $X$ and $Y$. Thus all the weights are congruent by $2$ because we have the cyclic module structure. 
\\\indent Since $V$ is finite dimensional, assume the maximal weight to be $n$, and its corresponding vector to be $v$. The weight space must be invariant under the Weyl group, which actions by reflection, so the lowest weight must be $-n$. 
\lemma We now claim that the sequence $\{v,Y(v),...\}$ generated $V$. 

\begin{proof} Indeed we only need to verify that it is invariant under the action of $sl_{2}\mathbb{C}$. For $Y^{m}(v)$ in the sequence we have
$$Y(Y^{m}(v))=Y^{m+1}v, H(Y^{m}(v))=(n-2m)Y^{m}(v)$$ and $$X(Y^{m})(v)=[X,Y]Y^{m-1}(v)+Y(X(Y^{m-1}(v)))=H(Y^{m-1}(v))+Y(X(Y^{m-1}(v)))$$ 

Since $H(Y^{m-1}(v))=(n-2m+2)Y^{m-1}(v)$, and for $m=1, Y(X(v))=0$, proceed by induction we may assume the general case to be $$X(Y^{m}(v))=m(n-m+1)Y^{m-1}(v)$$ Now assume this works for $m-1$, we have:$X(Y^{m-1}(v))=(m-1)(n-m+2)Y^{m-2}v$, for $m$ the above calculation gives $$X(Y^{m}(v))=(n-2m+2)Y^{m-1}(v)+(m-1)(n-m+2)Y^{m-1}(v)\\
=m(n-m+1)Y^{m-1}(v)$$ which verified our claim. 
\end{proof}
\theorem The irreducible representations of $sl_{2}$ is classified by their highest integer weight. The representation $V^{n}$ with highest weight $n$ is $n+1$ dimensional symmetric around the origin with weights $\{n,n-2,..-n\}$.
\begin{proof} Now if $m$ is the smallest power of $Y$ annihilating $v$, then from the relation above we have $0=X(Y^{m})(v)=m(n-m+1)Y^{m-1}(v)$. Since $Y^{m-1}(v)\not=0$, we have $n=m-1$. In particular $n$ is an integer. Hence the representation $V^{n}$ with highest weight $n$ is $n+1$ dimensional symmetric around the origin with weights $\{n,n-2,..-n\}$. This coincide with our reasoning with the Weyl group.
\end{proof}

We now carry out our first example:

\lemma The representation ring of $sl_{2}\mathbb{C}$ is isomorphic to $\mathbb{Z}[t+t^{-1}]$, the power series of $w=t+t^{-1}$ in integer coefficients.  

\begin{proof}
Let $V_{1}$ denotes the standard representation of $sl_{2}\mathbb{C}$ on $\mathbb{C}^{2}$. It can be decomposed as $\mathbb{C}v_{1}\oplus \mathbb{C}v_{-1}$ with $v_{1}=
\left(\begin{array}{ccc}
1  \\
0 \\
\end{array}\right)$, and $v_{-1}=
\left(\begin{array}{ccc}
0  \\
1 \\
\end{array}\right)$. They are eigenspaces of $H$ of eigenvalue $1$ and $-1$ respectively. The proof would be immediate if we show the map extending $V_{1}\cong \mathbb{C}v_{1}\oplus \mathbb{C}v_{-1}\mapsto t+t^{-1}$ is an isomorphism from $R(sl_{2}\mathbb{C})$ to $\mathbb{Z}[t+t^{-1}]$. To do that we need to show every element in $R(sl_{2})$ is generated by $V_{1}$ so we have a unique extension, then the isomorphism would naturally follows. 

Generating elements of $R(sl_{2}\mathbb{C})$ are of the form $\sum e_{n}V_{n}$, where $e_{n}V_{n}$ denotes the irreducible representation with highest weight $n$ and multiplicity $e_{n}\in \mathbb{Z}$. Assume $V_{n}$'s highest weight vector is $v_{n}$, and $V_{m}$'s highest weight vector is $v_{m}$. Then $v_{n}\otimes v_{m}$ have highest weight $n+m$. Consider the fact that $V_{n}$ is spanned by $v_{n},v_{n-2}...v_{-n}$, the eigenvalue and multiplicity of vectors of the form $v\otimes w, v\in v_{n},w\in v_{m}$ can be explicitly calculated. Thus we may decompose $V_{n}\otimes V_{m}$ as direct sum of $V_{i}$s quite explicitly. 

As the most basic example consider $V_{1}^{\otimes^{n}}$. We know $V_{1}$ is generated by $v_{1},v_{-1}$, and $V_{1}\otimes V_{1}$ is generated by $v_{2},v_{-2},v_{0}$ (hence isomorphic to $V_{2}$), further $V_{1}\otimes V_{1}\otimes V_{1}$ is generated by $v_{3},v_{-3},{\bf{v_{1}, v_{-1}}}$(appear twice, hence isomorphic to $V_{3}\oplus V_{1}$). Assume $V_{n}$ can be generated by $V_{1}$, we have $$V_{n}\otimes V_{1}=V_{n+1}\oplus V_{n-1}$$ Thus every element of $R(sl_{2}\mathbb{C})$ is generated by $V_{1}$ and the map we constructed is an isomorphism.  
\end{proof}

\remark It is obvious that the map we established in here is just the Char (stands for character) map introduced earlier in section 1. 

\lemma The representation ring of $H\in sl_{2}\mathbb{C}$ is isomorphic to $\mathbb{Z}[t,t^{-1}]$, the Laurent series of $t$ in integers. In here $
H=
\left(\begin{array}{ccc}
1 & 0 \\
0 & -1 \\
\end{array}\right)$.
\begin{proof}
Suppose $\rho:H\rightarrow gl(V)$ be a representation of $H$. We know$H^{2}=1$.Therefore $\rho(H)\cdot \rho(H)=1$ as well. Thus $\rho(H)$ must be either a reflection or a rotation along the axis in $\mathbb{C}^{n}$, or the identity matrix with suitable base changes. if $\rho(H)$ is a reflection then it fixes a $n-1$ dimensional subspace isomorphic to $\mathbb{C}^{n-1}$ and reflects its orthgonal complement by $-1$. Thus in this case $V\cong W\oplus W_{-1}$, with $W$ being the trivial representation $H(w)=w$, and $W_{-1}$ generated by eigenvector $w_{-1}$. If $\rho(H)$ is a rotation along the axis, then it fixes a $n-2$ dimensional subspace and sends its 2-dimensional orthogonal complement to its negative. Thus $V\cong W\oplus W_{-1}\oplus W_{-1}$. Hence the generating element in the representation ring is of the form $W\bigoplus W_{-1}$, with $\rho_{w}\cong I$ and $\rho_{w_{-1}}\cong -I$. 

The standard representation of $H$ on $\mathbb{C}^{2}$ may be decomposed as two one-dimensional representations $\mathbb{C}v_{1}$ and $\mathbb{C}v_{-1}$ respectively. In here the $v_{1}$ and $v_{-1}$ may be realized explicitly as $v_{1}=
\left(\begin{array}{ccc}
1  \\
0 \\
\end{array}\right)$;and $v_{-1}=
\left(\begin{array}{ccc}
0 \\
1 \\
\end{array}\right)$. Suppose $V$ is a representation of $H$. Since $H$ is abelian, the diagonalization of $V$ give $V=\bigoplus V_{\alpha}$. \\
\indent As a vector space $H$ is one dimensional. Thus $H^{*}$ is generated by the element $\alpha(H)=1$. The Weyl group reflects around origin, so $-\alpha(H)=-1$ is also a weight. Hence we see $V=W_{1}^{a}\oplus W_{-1}^{b}$, with $W_{1}=\mathbb{C}v_{1}$, $W_{-1}=\mathbb{C}v_{-1}$. Compare this with the result above we conclude that $W=W_{1}^{a}$ and $W_{-1}=W_{-1}^{b}$(the notation may be confusing to the reader). Hence the generating elements for $R(H)$ are $W_{1}$ and $W_{-1}$. We now map $W_{1}$ to $t$, $W_{-1}$ to $t_{-1}$. This is obviously an additive homomorphism. We claim it is also a multiplicative homomorphism. Hence by linearity it is an isomorphism. But this is obvious. 
\end{proof}

We now prove the main theorem of this section:

\theorem $R(H)$ is a free module over $R(sl_{2}\mathbb{C})$. In other words $\mathbb{Z}[t,t^{-1}]$ is a free module over $\mathbb{Z}[t+t^{-1}]$. 
\begin{proof}
We wish to prove $\mathbb{Z}[t,t^{-1}]\cong \mathbb{Z}[t+t^{-1}]\oplus t\mathbb{Z}[t+t^{-1}]$. For $t$ and $t^{-1}$ this is simply
$$t=t(1),1\in \mathbb{Z}[t+t^{-1}];t^{-1}=[t+t^{-1}]-t$$ 

We now prove the statement via induction on the degree of $x^{i}$. It is clear that if we can prove every monomial of the form $x^{i},i\in \mathbb{Z}$ is in $\mathbb{Z}[t+t^{-1}]\oplus t\mathbb{Z}[t+t^{-1}]$, then we can show every polynomial of the form $\sum a_{i}t^{i}$ is in $\mathbb{Z}[t,t^{-1}]\cong \mathbb{Z}[t+t^{-1}]\oplus t\mathbb{Z}[t+t^{-1}]$ as well. Suppose we have $t^{i}=f_{i}(t+t^{-1})+tg_{i}(t+t^{-1})$. Then the map $$\sum a_{i}t^{i}\rightarrow \sum a_{i}f_{i}(t+t^{-1})+ta_{i}g_{i}(t+t^{-1})$$ must be an isomorphism between $\mathbb{Z}[t,t^{-1}]\cong \mathbb{Z}[t+t^{-1}]\oplus t\mathbb{Z}[t+t^{-1}]$. We now prove that $x^{i}$ can be written in the form $t^{i}=f_{i}(t+t^{-1})+tg_{i}(t+t^{-1})$ as above. 

Assume for $|i|\le k$ the statement holds. For $t^{k+1}$ and $t^{-(k+1)}$ we have:
$$t(t+t^{-1})^{k}-t^{k}=\sum C^{k}_{i}t^{i}t^{i-k}=\sum C^{k}_{i}t^{2i-k},i\le k.$$
We may now apply the induction hypothesis, which gives us a new equation of the form $t(t+t^{-1})^{k}-t^{k}=f'(t+t^{-1})+g'(t+t^{-1})$. This showed the desired relation for $t^{k}$. Now for $t^{-(k+1)}$ we have:
$$(t+t^{-1})^{k+1}-t^{-(k+1)}=\sum C^{k+1}_{i=1}t^{i}t^{-k-1+i}-t^{-(k+1)}.$$
The highest power term in here is $t^{k+1}$. But by previous analysis it can be written in basis $(1,t)$ of the ring $\mathbb{Z}(t+t^{-1})$. All the rest terms has highest order less than $k$ and we can apply the induction hypothesis to them. Thus $t^{-k-1}$ can be written in terms of the basis $(1,t)$ as well. This finished the proof. 
\end{proof}

\discussion The above proof is very opaque in the sense it lacks the geometrical picture. Geometrically the above operation amounts to 'cancel' out positive/negative weights in the expansion of $(t+t^{-1})^{i}$ by a shifting of degree 1. It is not hard to imagine this process geometrically, but it is not clear how it generalizes in high degrees. The algebraic proof 'completed' the picture in this sense. 

\section{$n=3$}
\discussion The Lie algebra corresponding to $SU(3)$ is $\mathfrak{sl}_{3}(\mathbb{C})$, the set of all traceless 3-dimensional matrices. Its maximal solvable algebra $H$ is generated by $$
h_{1}\rightarrow
\left(\begin{array}{ccc}
1 & 0 & 0\\
0 & -1& 0\\
0 & 0 &0
\end{array}\right),
h_{2}\rightarrow
\left(\begin{array}{ccc}
0 & 0 & 0\\
0 & 1& 0\\
0 & 0 &-1
\end{array}\right),h_{3}=h_{1}+h_{2}\rightarrow
\left(\begin{array}{ccc}
1 & 0 & 0\\
0 & 0& 0\\
0 & 0 &-1
\end{array}\right)$$
The rest of the generators of $\mathfrak{sl}_{3}(\mathbb{C})$ are:
$x_{1}\rightarrow
\left(\begin{array}{ccc}
0 & 1 & 0\\
0 & 0& 0\\
0 & 0 &0
\end{array}\right),
x_{2}\rightarrow
\left(\begin{array}{ccc}
0 & 0 & 1\\
0 & 0& 0\\
0 & 0 &0
\end{array}\right)\\
x_{3}\rightarrow
\left(\begin{array}{ccc}
0 & 0 &0 \\
0 & 0& 1\\
0 & 0 &0
\end{array}\right),
x_{4}\rightarrow
\left(\begin{array}{ccc}
0 & 0 & 0\\
1 & 0& 0\\
0 & 0 &0
\end{array}\right),
x_{5}\rightarrow
\left(\begin{array}{ccc}
0 & 0 & 0\\
0 & 0& 0\\
1 & 0 &0
\end{array}\right),
x_{6}\rightarrow
\left(\begin{array}{ccc}
0 & 0 & 0\\
0 & 0& 0\\
0 & 1 &0
\end{array}\right)$. 

The roots of the Cartan decomposition are linear functions on $H$ such that $V=H\bigoplus V_{\alpha}$, $[h,v_{\alpha}]=\alpha(h)v_{\alpha}$. We now consider this decomposition in detail. Any element in $H$ of the form $\left(\begin{array}{ccc}
a & 0 & 0\\
0 & b& 0\\
0 & 0 &-a-b
\end{array}\right)$ can be decomposed as $ah_{1}+(a+b)h_{2}$. Assume eigenfunctions $e_{1}(ah_{1}+bh_{2})=a$, and $e_{2}(ah_{1}+bh_{2})=b$. They are just the dual basis for $h_{1}$ and $h_{2}$. 

Computation shows the following:
$$x_{1}\rightarrow 2e_{1}-e_{2}$$ $$x_{4}\rightarrow -2e_{1}+e_{2}$$ $$x_{2}\rightarrow e_{1}+e_{2}$$ $$x_{5}\rightarrow -e_{1}-e_{2}$$ $$x_{3}\rightarrow -e_{1}+2e_{2}$$ $$x_{6}\rightarrow e_{1}-2e_{2}$$ 
As we expected the root space is generated by $2e_{1}-e_{2}$, $e_{1}+e_{2}$. Further we can see three copies of $\mathfrak{sl}_{2}\mathbb{C}$ in $\mathfrak{sl}_{3}\mathbb{C}$; $\{x_{1},x_{4},h_{1}\},\{x_{3},x_{6},h_{2}\},\{x_{2},x_{5},h_{3}\}$ spans them. Further, $[x_{4},x_{6}]=x_{5}$. 
\begin{lemma}Let $V$ be an irreducible representation for $\mathfrak{sl}_{3}\mathbb{C}$. Then a highest weight vector $v$ for $H$ exists in $V$ in the sense there is no other vector whose weight's difference with $v$'s is a difference of $2e_{1}-e_{2}$ and $e_{1}-2e_{2}$. 
\end{lemma}
\begin{proof}
 We claim a highest weight vector as such exists. We may reason as follows: In case $v$ is not the highest eigenvector for $h_{2}$, we can {\bf{shift}} $v$ to higher eigenvalue by mutiplying $x_{2}$ (which shift by $2e_{1}-e_{2}$). For $h_{1}$ we may carry out a similar process. This process has to terminate at some $v'$. In case there are other weights being highest weights as well, $V$ would not be irreducible.

We also claim that $v$ must have integer eigenvalue {\bf{weight}} by the same argument we carried out in $\mathfrak{sl}_{2}\mathbb{C}$, which we shall not repeat in here. In particular $v$'s weight must be non-negative. 
\end{proof}
\begin{lemma}
Let $V$ be an irreducible representation of $\mathfrak{sl}_{3}\mathbb{C}$ and its highest weight $v$. Then $V$ is generated by the action of $\{x_{4},x_{6}\}$ on $v$. 
\end{lemma}
\begin{proof}
We claim that the action of $\{x_{4},x_{5},x_{6}\}$ on $v$ generates $V$. This is explicit since we may recover the representation of $\mathfrak{sl}_{3}\mathbb{C}$ by its three subrepresentations of $\mathfrak{sl}_{2}\mathbb{C}$, and each subrepresentation is generated by $\{x_{4},x_{6},x_{5}\}$ respectively. But $\{x_{4},x_{5},x_{6}\}$'s actions on $v$ in terms of weight lattice are not independent, we need only 2 of them. Without loss of generality we may assume they are $x_{4}$ and $x_{6}$ (since they are $h_{1}$ and $h_{2}$'s shifting vector in respective $\mathfrak{sl}_{2}\mathbb{C}$). 
\end{proof}
\begin{lemma}
$\mathfrak{sl}_{3}\C$ has Weyl group isomorphic to $S_{3}\cong D_{3}$. 
\end{lemma}
\begin{proof}
The adjoint representation on $V$ gives us fundamental roots as $2e_{1}-e_{2}$, $e_{1}+e_{2}$, $e_{1}-2e_{2}$. They constitute a root system. The fact that $W\cong S_{3}$ is immediate if we notice any reflection is automatically a permutation. For our purpose we wish to compute the Weyl group explicitly. We have the formula $$W_{\alpha}(\beta)=\beta-\frac{2k(\beta,\alpha)}{k(\alpha,\alpha)}\alpha$$ from \cite{Fulton}. Notice $\frac{2k(\beta,\alpha)}{k(\alpha,\alpha)}=\beta(H_{\alpha})$. For $\beta=(m,n)$ in the lattice, $\alpha=\{(2,-1),(-1,2),(1,1)\}$, $h_{(2,-1)}=h_{1}, h_{(-1,2)}=h_{2}, h_{(1,1)}=h_{3}$. We have $$W_{(2,-1)}(\beta)=(-m,m+n);W_{(2,-1)}(\beta)=(m+n,-n);W_{(1,1)}(\beta)=(-n,-m).$$ Thus the six weights generated by $W_{\alpha}$ are:
$$(m,n),(-m,m+n),(m+n,-n),(n,-m-n),(-n,-m),(-m-n,m).$$ We may verify explicitly that the six element group is isomorphic to $D_{3}$. 
\end{proof}
\example Our selection of the basis is `skewed'. This is apparaent from the picture: Part of the weight diagram for $V$ with highest weight $(2,2)$ is:
\begin{center}
\includegraphics{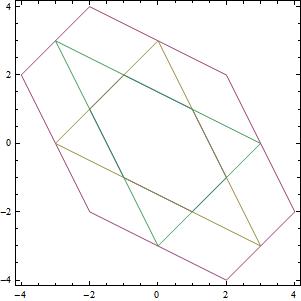}
\end{center}
\discussion We now wish to examine the first few irreducible representations we know. The standard representation of $\mathfrak{sl}_{3}\mathbb{C}$ acting on $\mathbb{C}^{3}$ is irreducibleIt is clear from the above picture that we have irreducible representation $Q_{1}$ and $Q_{2}$ with weights $(1,0),(-1,1),(0,-1)$, $(1,-1),(0,1),(1,-1)$ respectively.

\lemma Define $Q_{1}=V_{(1,0)}$ and $Q_{2}=V_{(0,1)}$, we have $V_{(0,n)}\otimes V_{(0,1)}\cong V_{(0,n+1)}\oplus V_{(1,n-1)}$. 
\begin{proof}
We claim that $V_{(0,n)}\otimes V_{(0,1)}\cong V_{(0,n+1)}\oplus V_{(1,n-1)}$. This is clear because by our definition earlier we cannot shift $(0,n)$ to `lower' weights by applying $x_{4},x_{5},x_{6}$. Thus the intergral weights for $V_{(0,n)}$ are $(0,n),(n,-n),(-n,0)$. Considering $V_{(0,1)}$ has weights $(0,1),(1,-1),(-1,0)$, adding them and counting the multiplicity give us the desired statement.
\end{proof}  
\lemma The representation ring of $\mathfrak{sl}_{3}\mathbb{C}$ is isomoprhic to $\mathbb{Z}[t_{1}+t_{1}^{-1}t_{2}+t_{2}^{-1},t_{1}t_{2}^{-1}+t_{2}+t_{1}^{-1}]$. 
\begin{proof}
We wish to show that $R(\mathcal{G})$ is generated by fundamental representations $Q_{1}$ and $Q_{2}$. This is explicit if we can show every irreducible representation with highest weight $(m,n)$ is contained in $Q_{1}^{m}\otimes Q_{2}^{n}$. I acknkowledge honestly that it is difficult to write out a general formula for the multiplicity of individual weights in the tensor product. I suspect we have the result:
$$V_{(m,n)}\otimes Q_{1}=V_{(m+1,n)}\oplus V_{(m-1,n+1)}\oplus V_{(m,n-1)},V_{(m,n)}\otimes Q_{2}=V_{m,n+1}\oplus V_{(m+1,n-1)}\oplus V_{m-1,n}$$
But it is generally difficult to prove without using elementary methods like induction. We can `bypass' this issue by claiming this $Q_{1}^{m}\otimes Q_{2}^{n}$ must contain a copy of $V_{(m,n)}$, which is evident since every individual weight space of $V_{(m,n)}$ is generated by the tensor product. Now we map $Q_{1}$ to $t_{1}+t_{1}^{-1}t_{2}+t_{2}^{-1}$, and $Q_{2}$ to $t_{1}+t_{1}^{-1}t_{2}+t_{2}^{-1}$ by mapping $t_{1}$ and $t_{2}$ to be one-dimensional representation of $\mathfrak{sl}_{3}\mathbb{C}$ with weights $(0,1)$ and $(1,0)$ respectively. 
\end{proof}

\lemma The representation ring of $H=\{h_{1},h_{2}\}$ is isomorphic to $\mathbb{Z}[t_{1},t_{1}^{-1},t_{2},t_{2}^{-1}]$. 
\begin{proof}
We know $H$ is the maximal solvable subalgebra inside of $\mathfrak{sl}_{3}\C$. Any representation $V$ of $H$ can be diagnoalized since $H$ is abelian. Thus we have $V=\bigoplus V_{\alpha}$, where $H(v)=\alpha(H)v$. 
\\\indent If $V$ is irreducible, then it serves as a generator in $R(H)$, and all $V_{\alpha}$ is one dimensional. Now notice the standard representation of $H$ on $\mathbb{C}^{3}$ can be decomposed as $\mathbb{C}e_{1}\oplus \mathbb{C}e_{2}\oplus \mathbb{C}e_{3}$, with $e_{i}$ the standard basis vectors. Consider the adjoint action $e_{i}$ to $h_{i}$, we have $e_{1}$ having eigenvalue $(1,0,1)$, $e_{2}$ having eigenvalue $(-1,1,0)$, $e_{3}$ having eigenvalue $(0,-1,1)$. Using linear combination of these vectors, it is not difficult to construct eigenvectors that have eigenvalue $1$ and $-1$ for $h_{1}$, while vanishing for $h_{2}$; and vice-versa we may construct such for $h_{2}$. We now list them as follows:
\begin{enumerate}
\item For $h_{1}$: $$h_{1}(1,0,0)=(1,0,0)$$ $$h_{1}(0,1,1)=-(0,1,1)$$ $$h_{2}(1,0,0)=h_{2}(0,1,1)=0$$
\item 
For $h_{2}$: $$h_{2}(0,0,1)=(0,0,1)$$ $$h_{2}(1,1,0)=-(1,1,0)$$ $$h_{1}(0,0,1)=h_{1}(1,1,0)=0$$
\end{enumerate} 
If we name them to be $V_{(1,0)}$, $V_{(-1,0)}$, $V_{(0,1)}$, $V_{(0,-1)}$, then we claim $R(H)$ is in fact generated by the four elements. Take consideration that the weight of the tensor product is the sum of the individual weights, this is clear since each irreducible representation is generated by the $V_{(i,j)}$, and the irreducible elements generated the representation ring. Thus $R(H)\cong \mathbb{Z}[t_{1},t_{1}^{-1},t_{2},t_{2}^{-1}]$. 
\end{proof}

\theorem  $R(H)$ is a free module over $R(\mathfrak{sl}_{3}\mathbb{C})$ with basis $$\{1, t_{1},t_{2},t_{1}t_{2}^{-1},t_{1}^{-1}t_{2}, t_{1}t_{2}\}$$
\begin{proof}
This amounts to prove that $R_{1}=\mathbb{Z}[t_{1},t_{1}^{-1},t_{2},t_{2}^{-1}]$ is a free module over $R_{2}=\mathbb{Z}[t_{1}+t_{1}^{-1}t_{2}+t_{2}^{-1},t_{1}t_{2}^{-1}+t_{2}+t_{1}^{-1}]$. We wish to prove this statement by finding an explicit basis. I now claim:
$$R_{1}\cong 1*R_{2}\oplus t_{1}R_{2}\oplus t_{2}R_{2}\oplus t_{1}t_{2}^{-1}R_{2}\oplus t_{1}^{-1}t_{2}R_{2}\oplus t_{1}t_{2}R_{2}$$

We shall only briefly capitulate the proof resembling the case of $n=2$ here (because it is extremely messy) and construct a much cleaner proof in a more canoical basis in the section below. The main idea of the messy proof is to use induction on $k$ of monomials of the form $t_{1}^{i}t_{2}^{j}$ with $|i+j|\le k, \text{max}(|i|,|j|)\le 2k$. Assuming for $k\le K$ this works, we may prove the case for $k=K+1$ by constructing $t_{1}^{i}t_{2}^{j}$ by $(t_{1}+t_{1}^{-1}t_{2}+t_{2}^{-1})^{i}\cdot (t_{1}+t_{1}^{-1}t_{2}+t_{2}^{-1})^{j}$ and canceling of the lower degree monomials using the induction hypothesis and multiplication of basis elements in the same way we did for the $n=2$ case. For example for $i=2,j=2$ using symbols $x=t_{1},y=t_{2}$ we have:

\begin{align*}
(x+y^{-1}+x^{-1}y)^{2}(xy^{-1}+y+x^{-1})^{2}=
15+\frac{2}{x^3}+2 x^3+\frac{x^2}{y^4}+\frac{2}{y^3}+\frac{2 x^3}{y^3}+{\bf{\frac{1}{x^2 y^2}}}+\frac{8 x}{y^2}\\
+\frac{x^4}{y^2}+\frac{8}{x y}+\frac{8 x^2}{y}+\frac{8 y}{x^2}+8 x y+\frac{y^2}{x^4}+\frac{8 y^2}{x}+x^2 y^2+2 y^3+\frac{2 y^3}{x^3}+\frac{y^4}{x^2}
\end{align*}

It is obvious that the other terms all have lower degree so we can apply the induction hypothesis. We comment that in general it is $\bf{not}$ desirable to use such method to prove $R(\mathfrak{sl}_{n}\mathbb{C})$ is a free module over $R(H)$ because of two reasons: one is the selection of a basis of $R(H)$ in terms of $R(\mathfrak{sl}_{3}\C)$ is usually non-canonical and may be arbitrary unless we introduce some (e.g, monomial) ordering of the elements; the other is the above construction lacks the geometric picture of the action of Weyl group on the Lie group and its maximal torus, hence could not provide much desirable information of the structure of the representation ring. 
\end{proof}
\section{$n=3$ and the general approach}
\discussion In this section we follow Serre's treatment in \cite[Chapter VII]{Serre}, Much of the material can also be found in \cite{Fulton}, Chapter 23 and preceding chapter discussing representation of $\mathfrak{sl}_{n}$, the author tried to follow their notation. It is not clear to me how much is known to experts (thus I cannot give an accurate bibliography) despite a reference request in Mathoverflow. The author is grateful for David Speyer's explicit answer, which fully characterizes the algebraic nature of this problem. 
\remark In the previous case we used the `skewed' basis of $H$ as $h_{3}=h_{1}+h_{2}$ (which we ignored). If we view $SU_{n}$ as part of the group $U_{n}$ whose maximal torus having $(z_{1},...z_{n})$ on the diagonal, we may attempt to find a more canonical basis. Indeed this is the approach adopted in \cite{Bourbaki} and \cite{Serre}. We now give a slightly different proof of $n=3$ case and attempt to generalize the situation to any finite dimension $sl_{n}$. It is not clear at the time of writing if this proof would be original.  
\definition $H_{i}$ is the diagonal matrix with 1 on $(i,i)$th position and 0 elsewhere. It is obvious that $H_{i}$ does not belong to $h$ but pairwise difference $H_{i}-H_{j}$ is. Its matrix multiplication takes $e_{i}$ (n-vector having $i$th entry to be 1, the rest are 0) to itself and kills $e_{j}$ for $j\not=i$.
\remark The Cartan subalgebra is equal to $$h=\{a_{1}H_{1}+a_{2}H_{2}....+a_{n}H_{n};a_{1}+a_{2}+...a_{n}=0\}$$
\definition Define $L_{i}\in h^{*}$ to be $L_{i}(H_{i})=\delta_{i,j}$. Then $h^{*}=\mathbb{C}\{L_{1},L_{2}...,L_{n}\}/(L_{1}+L_{2}+...+L_{n}=0)$. 
\remark If $E_{(i,j)}$ is the matrix having $1$ at $(i,j)$ and $0$ in the rest entries, then we can view it as the endomorphism of $\mathbb{C}^{n}$ carrying $e_{j}$ to $e_{i}$ and killing $e_{k}$if $k\not=j$. We then have $$ad(a_{1}H_{1}+a_{2}H_{2}....+a_{n}H_{n})E_{i,j}=(a_{i}-a_{j})E_{(i,j)}$$ In other words, $E_{i,j}$ is the eigenvector with eigenvalue $L_{i}-L_{j}$. As a consequence the roots of $sl_{n}$ are just pair $L_{i}-L_{j}$. 
\\\indent The following discussion is not directly related to our final proof:
\lemma The Killing form on $sl_{n}$ is given by $$k\left(\sum a_{i}H_{i},\sum b_{i}H_{i}\right)=2n\sum a_{i}b_{i}$$  and $$k\left(\sum a_{i}L_{i},\sum b_{i}L_{i}\right)=\frac{1}{2n}\left(\sum_{i}a_{i}b_{i}-\frac{1}{n}\sum_{i,j}a_{i}b_{j}\right)$$ 

\begin{proof}
We follow the proof given by \cite{Fulton}. We note the automorphism $\phi$ of $C^{n}$ interchanging $e_{i}$ and $e_{j}$ and fixing $e_{k}$ for all $k\not=i,j$ induces an automorphism $Ad(\phi)$ of the Lie algebra $sl_{n}\mathbb{C}$ that carries $h$ to itself and exchanging $H_{i}$ and $H_{j}$. The Killing form on $h$ must be invariant under all such automorphisms, hence we must have $k(L_{i},L_{j})=k(L_{j},L_{i}),k(L_{i},L_{k})=k(L_{j},L_{k})$ for all $i,j,k\not=i,j$. Thus the Killing form on $h^{*}$ must be a linear combination of the form $\sum a_{i}b_{i}$ and $\sum_{j\not=i}a_{i}b_{j}$. Considering that $k(\sum a_{i}L_{i},\sum b_{i} L_{i})=0$ whenever all $a_{i}=a_{j}$ or all $b_{i}=b_{j}$, the Killing form is a multiple of $$\sum_{i}a_{i}b_{i}-\frac{1}{n}\sum_{i,j}a_{i}b_{j}$$

To calculate $k(\sum a_{i}H_{i},\sum b_{i}H_{i})$ we argue as follows: if $x,y\in h$ and $Z_{\alpha}$ generates $g_{\alpha}$, then $ad(x)ad(y)(Z_{\alpha})=\alpha(x)\alpha(y)Z_{\alpha}$. So $k(x,y)=\sum_{\alpha\in \Phi} \alpha(x)\alpha(y)$. For $\mathfrak{sl}_{n}$ the roots are pairwise differences $L_{i}-L_{j}$. Thus we have $$k\left(\sum a_{i}H_{i},\sum b_{i}H_{i}\right)=\sum_{i\not=j}\left(a_{i}-a_{j}\right)(b_{i}-b_{j})=\sum_{i}\sum_{j\not=i}\left(a_{i}b_{i}+a_{j}b_{j}-a_{i}b_{j}-a_{j}b_{i}\right)$$
This can be further simplified by noting $\sum a_{i}=0,\sum b_{i}=0$ to $$k\left(\sum a_{i}h_{i}\sum b_{i}h_{i}\right)=2n\sum a_{i}b_{i}.$$ Now note that we have $k(\alpha,\beta)=k(t_{\alpha},t_{\beta}),t_{\alpha}=\frac{2h_{\alpha}}{k(h_{\alpha},h_{\alpha})}$. Thus $$k\left(\sum a_{i}L_{i},\sum b_{i}L_{i}\right)=\frac{1}{2n}[\sum_{i}a_{i}b_{i}-\frac{1}{n}\sum_{i,j}a_{i}b_{j}].$$
We may use the Killing form the calculate the Weyl group explicitly, but there is a more direct computation:
\end{proof}

\lemma The {\bf{Weyl group}} of $sl_{n}$ is isomorphic to $S_{n}$. 
\begin{proof}
The Weyl group is generated by the reflections of the form $W_{\alpha}$, $\alpha\in h$. Now we have $W_{\alpha}(\beta)=\beta-\frac{2k(\beta,\alpha)}{k(\alpha,\alpha)}\alpha$. Thus assume $\beta=\sum a_{i}L_{i},\alpha=L_{j}-L_{k}$, then we have $W_{\alpha}(\beta)=\sum a_{i}L_{i}-\beta(H_{j}-H_{k})(L_{j}-L_{k})=\sum_{i\not=j,k} a_{i}L_{i}+a_{j}L_{k}+a_{k}L_{j}$. This showed its action is just a transposition between $\alpha$ and $\beta$ while fixing the other elements. Thus it must be isomorphic to $S_{n}$. 
\end{proof}

\lemma We claim that $R(h)\cong \mathbb{Z}(x_{1},x_{2}....,x_{n})|\prod x_{i}=1\cong \mathbb{Z}(x_{1},x_{1}^{-1}...x_{n},x_{n}^{-1})$. 
\begin{proof}
This is obvious is we translate the condition of operations of representation rings with their roots, we have $\sum L_{i}=0\rightarrow \prod x_{i}=1$. 
\end{proof}

\lemma $R(\mathfrak{sl}_{n})\cong \mathbb{Z}[\sum x_{i},\sum_{i\le j} x_{i}x_{j},\sum_{i\le j\le k}x_{i}x_{j}x_{k}...]$. In other words the representation ring of the group is a polynomial ring in the symmetric polynomials. 
\begin{proof}
In general we assert that if $R(G)$ is reducible (in here we know because of Peter-Weyl theorem, in general we only need to know $G$ is compact), assume the characters of $T$ form a lattice $X$, then $R(T)\cong \mathbb{Z}[X]$, $R(G)\cong \Z[X]^{W}$ (the subring of $\Z[X]$ that is invariant under the action of $W$). But we already proved this in the introduction section; thus we only need to verify that the symmetric polynomials are invariant under the action of the Weyl group, which is trivial. 
\end{proof}

\theorem $R(h)$ is a free module over $R(\mathfrak{sl}_{3})$.
\begin{proof}
We \footnotetext{We proceed with a proof suggested by Prof. James Belk.} identify $$A=\Z[x+y+z,xy+yz+zx]/(xyz-1).$$ $$B=\Z[x,y,z]/(xyz-1).$$ In general we have a sequence of module extensions:
$$A\subset A[x]\subset A[x,y]\subset A[x,y,z]=B$$
If we can show that each extension in the sequence is a free module over the previous one, then we can assert $B$ must be a free module over $A$ since then for $b\in B$, we would have $b=\sum b_{i}A_{i},A_{i}\in A[x,y]$, and we can decompose $A_{i}$ even further by using $A_{j}\in A[x]$ and $A_{k}\in A$. 

We now consider an auxillary polynomial $$P(t)=t^{3}-(x+y+z)t^{2}+(xy+yz+xz)t-1$$ Obviously $P(t)=\prod (t-x_{i})$. Since $x$ is a root of $P(t)$, we found $$x^{3}=(x+y+z)x^{2}-(xy+yz+xz)x+1$$ Thus $A[x]$ is a free over $A$ by the basis $[1,x,x^{2}]$, since for any element of the form $ax^{k},s\in A$ we may decompose $x^{k}$ into $x^{k}=\sum^{2}_{i=0} a_{i}x^{i},a_{i}\in A$. Thus $ax^{k}=\sum^{2}_{i=0} (aa_{i})x^{i}$. 

We still need to prove $A[x,y]$ is free over $A[x]$. We consider the auxillary polynomial ring $A^{*}=\Z[y+z,yz]$. Then we have $y+z=x+y+z-(x),yz=xy+yz+zx-x(y+z)$. It then follows $A[x]=A^{*}[x]$. Now since $y$ satisfies the minimal polynomial $t^{2}-t(y+z)+yz=0$ with coefficients in $A^{*}[x]$, thus $A^{*}[x,y]=A[x,y]$ is free over $A^{*}[x]=A[x]$ with basis $(1,y)$. And the decomposition of monomial involving $y$ in $A[x,y]$ similarly follows. 

In the last step, we do not need to adjoin $z$ as $z=x+y+z-(x+y)$\footnotetext{Thus the basis we found for $\mathfrak{sl}_{3}$ coincide with the one provided by {\bf{David Speyer}}}.  $1,x,y,xy,x^{2},x^{2}y$, as it is the product of the basis in the intermediate extensions. Steinberg has proved in his paper that $R(T)$ is of $|W|=6$ over $R(G)$. Thus they indeed span $R(T)$ as an $R(G)$ module. 
\end{proof} 

\remark Using the same idea it is not difficult to construct a similar basis for the general $R(\mathfrak{sl}_{n})$:
\theorem $R(h)$ is a free module over $R(\mathfrak{sl}_{n})$. A standard basis is $\{\prod x_{i}^{a_{i}},0\le a_{i}\le n-i\}$.
\begin{proof}
In the general situation $$R(h)=\Z[x_{1},x_{2}..,x_{n}]/\left(\prod x_{i}-1\right).$$ $$R(G)=\Z[\sum x_{i},\sum_{i<j} x_{i}x_{j}..]/\left(\prod x_{i}-1\right).$$
Define $s_{i}$ to be the $i$th symmetric polynomial. Than analgous to above we may define $R(\mathfrak{sl}_{n})=A=[s_{1},s_{2}...s_{n}]$, and we obtain an analgous sequence: $$A\subset A[x_{1}]\subset A[x_{1},x_{2}]\subset A[x_{1},x_{2},x_{3}]...\subset A[x_{1}...x_{n-1}]=\Z[x_{1},x_{2}..,x_{n}]=R(h)$$ 
We may construct the same formal polynomial $$P(t)=t^{n}-(\sum x_{i})t^{n-1}+(\sum x_{i}x_{j})t^{n-2}....\pm 1=\prod (t-x_{i})$$
In here invert the equation we have $x_{1}^{n}=(\sum x_{i})x_{1}^{n-1}-(\sum x_{i}x_{j})x_{1}^{n-2}...\pm1$. Thus any monomial element $ax_{1}^{i},a\in A$ in $A[x_{1}]$ can be written in a unique way as $ax_{1}^{i}=\sum a^{n-1}_{j=0}a_{j}x_{1}^{j},a,a_{j}\in A$. 

We now proceed with induction assuming for $A[x_{1},x_{2}, \ldots, x_{k}]$ we have verified that $$A[x_{1}]\subset A[x_{1},x_{2},x_{3}],\ldots\subset A[x_{1}..x_{k}]$$ is a sequence of module extensions such that the right handed one is always a free module over its left handed one. Thus we need to prove $A[x_{1}..x_{k+1}]$ is free over $A[x_{1}..x_{k}]$. We now define the analgous $A_{k}^{*}$ to be $\Z[t_{1}..t_{n-k}]$, where $t_{i}$ is the $i$th symmetric polynomial in $x_{k+1},x_{k+2}..x_{n}$. We may factor out $\prod^{k}_{i=1} (t-x_{i})$ from $P(t)$ to obain a new polynomial $Q(s)=s^{n-k}-t_{1}s^{n-k-1}...+t_{n-k-1}s\pm t_{n-k}$. Since $x_{k+1}$ satisfy this polynomial, $A_{k}^{*}[x_{k+1}]$ is a rank $n-k$ free module over $A_{k}^{*}$. Further, since $\{x_{i},i\in \{1,..k\}\}$ are not solutions to this polynomial, this is also the minimal polynomial for $x_{k+1}$ over $A_{k}^{*}[x_{1},x_{2}...x_{k}]$. In other words $\{1,x_{k+1},...x_{k+1}^{n-k}\}$ is a basis for $A_{k}^{*}[x_{1},x_{2}...x_{k},x_{k+1}]$ over $A_{k}^{*}[x_{1},x_{2}...x_{k}]$. 

We now try to prove that $A_{k}^{*}[x_{k}]\cong A_{k-1}^{*}[x_{k}]$. We mark the symmetric polynomials in $A_{k-1}^{*}$ as $T_{i}$, and the symmetric polynomials in $A_{k}^{*}$ as $t_{i}$. Then we have $T_{i}=t_{i}+x_{k}t_{i-1}$. Thus $A_{k}^{*}[x_{k}]\cong A_{k-1}^{*}[x_{k}]$. Thus we may adjoin $\{x_{1},...x_{k}\}$ one by one to $A_{k}^{*}$ to obtain $A_{k}^{*}[x_{1},x_{2}...x_{k}]\cong A[x_{1},..x_{k}]$. Thus $A_{k}^{*}[x_{1},x_{2},...x_{k+1}]\cong A[x_{1},x_{2},..x_{k},x_{k+1}]$. Use the result we proved in previous passage, we obtain $A[x_{1},x_{2},..x_{k},x_{k+1}]$ is a free $A[x_{1},x_{2},..x_{k}]$ module with basis $\{1,x_{k+1},...x_{k+1}^{n-k}\}$.

We can thus obtain a general basis as any of the combinations of $\prod x_{i}^{a_{i}},0\le a_{i}\le n-i$. Its order, as we expected, is $n$!. 
\end{proof}

\section{Index theory methods}
\discussion In this section we go back to index theory of elliptic operators. 
\definition We begin this section by introducing a symmetric non-degenerate inner product from equivariant K-theory on $G/T$: 
$$\langle M_{1},M_{2}\rangle=(-1)^{sgn(w)}V_{w(M_{1}M_{2}),}-\rho, w\in W$$

In here $\rho$ is the 'sum of fundamental weights', equal to $\sum \lambda_{i}$. The value of the inner product is defined as follows: Consider the Weyl group action that move it to the positive Weyl chamber. If as a result $(-1)^{sgn(w)}V_{w(M_{1}M_{2}),}-\rho, w\in W=0$, then $\langle M_{1},M_{2}\rangle=1$; If instead $\langle M_{1},M_{2}\rangle$ onto the boundary, then $\langle M_{1},M_{2}\rangle=0$. Otherwise (neither on boundary nor on the smallest positive root), the inner product gives us a representation in $R(G)$as $V_{w(M_{1}M_{2})}$. Thus this is an inner product $R(T)\times R(T)\rightarrow R(G)$. \\
\indent My advisor Gregory Landweber has proved in his published paper \cite{Greg} that this is indeed non-degenerate on $R(T)\times R(T)$. Further reference may be found in \cite{Merkurjev}. Thus if we can find a basis whose matrix $e_{ij}=\langle M_{i},M_{j}\rangle$ under this basis is unimodular(having determinant 1 and integral coefficients), then the basis elements must be linearly independent, and checking their rank would suffice to let us prove $R(T)$ is a free module over $R(G)$ with that basis. Of course, we need to show that this inner product is well-defined on $R(T)$ as a module over $R(G)$ before we can make the above assertions. 

A similar strategy is to construct a dual element of every basis element via some linear combinations of the monomials. This turned out to be quite difficult, however. 

\begin{example} We list the matrix for the basis $\{1,x,y,xy,x^{2},x^{2}y\}$, first we list the product of basis elements: 
$$\begin{bmatrix}
M_{1}M_{2} & 1 & x & y & xy & x^{2} & x^{2}y \\[6.5pt]
1&   1& x &  y& xy & x^{2} & x^{2}y \\[6.5pt]
x&  x &  x^{2} & xy & x^{2}y & x^{3} & x^{3}y \\[6.5pt]
y & y &  xy     &  y^{2} & xy^{2} & x^{2}y & x^{2}y^{2}\\[6.5pt]
xy& xy &  x^{2}y     &  xy^{2} & x^{2}y^{2} & x^{3}y & x^{3}y^{2} \\[6.5pt]
x^{2} & x^{2} &  x^{3} & x^{2}y & x^{3}y & x^{4} & x^{4}y \\[6.5pt]
x^{2}y & x^{2}y & x^{3}y &x^{2}y^{2} &x^{3}y^{2}&x^{4}y&x^{4}y^{2}
\end{bmatrix}$$

The monomials $x^{i},y^{j}, (xy)^{k}$ are all on the Weyl boundary, which is the elements in the weight lattice invariant under the Weyl group action, thus they vanish. Further $x^{2}y=\rho$ implies $\langle 1, x^{y}\rangle=1$. 

Thus under the symmetric inner product $$(-1)^{sgn(w)}V_{w(M_{1}M_{2}),}-\rho, w\in W$$ the inner product matrix become 
$$\begin{bmatrix}
 \langle M_{1}, M_{2}\rangle& 1 & x & y & xy & x^{2} & x^{2}y \\[6.5pt]
1&   0& 0&  0& 0 & 0 & 1 \\[6.5pt]
x&  0 &  0 & 0 & 1 & 0 & x^{3}y \\[6.5pt]
y & 0 &  0     &  0 & -1 & 1 & 0\\[6.5pt]
xy& 0 &  1     &  -1 & 0 & x^{3}y & x^{3}y^{2} \\[6.5pt]
x^{2} &0 &  0 & 1 & x^{3}y & 0 & x^{4}y \\[6.5pt]
x^{2}y & 1 & x^{3}y & 0 &x^{3}y^{2}&x^{4}y&x^{4}y^{2}
\end{bmatrix}$$ It is unimodular because after computation it has determinant 1. 
\remark We conjecture that Steinberg's basis is orthonormal under the index theory inner product, but we have no means to prove that at this stage. My advisor has verified in $\mathfrak{sl}_{4}$ case Steinberg's basis is indeed orthonormal. 

\end{example}

\chapter{Work for other Lie groups}
\section{Representation for $so_{2n}$}
\section{n=1}
When $n=1$, $so_{2}$ are all matrices of the form 
$\left(\begin{array}{cc}
a & 0\\
0 & a^{-1}
\end{array}\right)$ since $M$ has to satisfy $\det(M)=1,M^{t}Q+Q M=0$. In here $Q=\left(\begin{array}{cc}
0 & 1\\
1 & 0
\end{array}\right)$. Consider $M=\left(\begin{array}{cc}
a & b\\
c & d
\end{array}\right)$, this gives us $ad-bc=1$, $(c,b)+(b,c)=0$. Thus $b=-c=0$, $ad=1$. In this case $so_{2}$ itself is abelian isomorphic to $\C^{*}$ (non-zero complex numbers). 

The representation ring $R(so_{2})$ therefore is isomorphic to $R(h)$. Any representation of $so_{2}$ is just homomorphism from $\C^{*}$ to $\C^{n}$, thus together they are isomorphic to $\C^{*}\otimes \C^{n}=\C^{n}$. It follows $R(so_{2})\cong \Z$. Therefore it follows trivially $R(\mathfrak{so}_{2})$ is a free module over $R(H)$ since they are the same. 
\section{n=2}
When $n=2$, $so_{4}$ are all rotations in the 3-dimensional sphere $\mathbb{S}^{3}$. We assert that $so_{4}\cong \mathfrak{sl}_{2}\times \mathfrak{sl}_{2}$. This can be done by the classical isomorphism identifying the unit quarterions in $\mathbb{H}$ to be isomorphic to $\mathfrak{sl}_{2}$, and consider the action $(q_{1},q_{2}),q_{i}\in \mathbb{H}$ on $p\in \mathbb{S}^{3}$ by $q_{1}pq_{2}^{-1}$. This provides an isomorphism from $\mathfrak{sl}_{2}\times \mathfrak{sl}_{2}\rightarrow so_{4}$ because every rotation in $\mathbb{S}^{3}$ is uniquely determined by $q_{1},p,q_{2}$, and though we have a $2\rightarrow 1$ map at the Lie group level, at the Lie algebra level the tangent space just coincide. 

From our analysis of $\mathfrak{sl}_{2}$ we know its representation ring is isomorphic to $\Z[x+x^{-1}]$, and the represenation ring of its maximal cartan subalgebra is isomorphic to $\Z[x,x^{-1}]$. Thus the representation ring of $so_{4}$ is isomorphic to $\Z[x+x^{-1}]\otimes \Z[y+y^{-1}]\cong \Z[x+x^{-1},y+y^{-1}]$ and $R(h)\cong \Z[x,x^{-1}]\otimes \Z[y,y^{-1}]=\Z[x,x^{-1},y,y^{-1}]$. 

\theorem $\Z[x,x^{-1},y,y^{-1}]$ is a free module over $\Z[y+y^{-1}]\cong \Z[x+x^{-1},y+y^{-1}]$.
\begin{proof}
We notice the extension of modules as follows:
$$\Z[x+x^{-1},y+y^{-1}]\subset \Z[x+x^{-1},y+y^{-1}][x]\subset \Z[x,x^{-1},y,y^{-1}]$$
We know $\Z[x+x^{-1},y+y^{-1}][x]$ is a free module over $\Z[x+x^{-1},y+y^{-1}]$ with basis $(1,x)$ because $\Z[x+x^{-1},y+y^{-1}][x]=\Z[x,x^{-1},y+y^{-1}]$. Assume every monomial of the form $x^{i},i\in \Z$ can be written as $1*(x+x^{-1})^{a_{i}}+x(x+x^{-1})^{b_{i}}$ (we proved this in the $\mathfrak{sl}_{2}$ section, it also comes straightforward from the relation $x^{2}=x(x+x^{-1})-1$), then every monomial of the form $x^{i}(y+y^{-1})^{j},i,j\in Z$ would decompose into $1*(x+x^{-1})^{a_{i}}(y+y^{-1})^{j}+x(x+x^{-1})^{b_{i}}(y+y^{-1})^{j}$. This gives the desired decomposition. 

We also know $\Z[x,x^{-1},y,y^{-1}]$ is free over $\Z[x,x^{-1},y+y^{-1}]$ by an analgous argument, and the basis is $\{1,y\}$. Thus the general basis is $\{1,x,y,xy\}$.
\end{proof}

\discussion The above proof is purely algebraic and non-constructive. It is not clear how to decompose $x^{i}y^{j}$ into $(x+x^{-1})$ and $(y+y^{-1})$ by the four basis elements. Geometrically, the two copies of $\mathfrak{sl}_{2}$ sits in the diagonal line of a square, and it is clear that the four basic directions are just $1,x,y,xy$. 

\section{n=3}
We do not see a similar phenomenon in $so_{6}$ right away, thus we need to re-visit the roots, weights, weyl group, etc. In the end we shall find $R(so_{6})\cong R(sl_{4})$ in a subtle way. The following discussion more or less comes from Fulton and Harris:

\begin{lemma}
The representation ring of the cartan subalgebra of $so_{6}$ is isomorphic to $\Z[x,y,z,x^{-1},y^{-1},z^{-1}]$.
\end{lemma}
\begin{proof}
This follows from the matrix decomposition of $M$ into the form $
\left(\begin{array}{cc}
A & 0\\
0 & D
\end{array}\right)$, and a short computation showed $A^{t}+D=D^{t}+A=0$. Therefore $M$ is essentially determined by the three diagonal elements in $A$. The rest follows since the abelian part of $SO(3)$ would be a torus isomorphic to $\mathbb{S}^{1}\times \mathbb{S}^{1}\times \mathbb{S}^{1}$. 
\end{proof}
As a result we have:
\begin{definition}
Define $E_{i,j}$ be the matrix with $1$ in $\{i,j\}$ position and 0 everywhere else as usual. The cartan subalgebra of $so_{6}$ is generated by $H_{1},H_{2},H_{3}$. $H_{i}=E_{i,i}-E_{3+i,3+i}$. $H_{i}$'s action on $\C^{6}$ is to fix $e_{i}$, send $e_{3+i}$ to its negative, and kill all the remaining vectors. We define the basis for the dual vector space $h^{*}$ by $L_{j}$ where $\langle L_{j},H_{i}\rangle=\delta_{i,j}$.
\end{definition} 
\remark This can be extended to the general $so_{2n}$ case. 

\begin{lemma}
$X_{i,j}=E_{i,j}-E_{n+j,n+i}$ is an eigenvector for the adjoint action of $h$ with eigenvalue $L_{i}-L_{j}$
\end{lemma}
\begin{proof}
We claim that $X_{i,j}=E_{i,j}-E_{n+j,n+i}$ is an eigenvector for the adjoint action of $h$ with eigenvalue $L_{i}-L_{j}$. We first notice that by definition we have $[H_{i},E_{i,j}]=E_{i,j}$ because $$[H_{i},E_{i,j}]=[E_{i,i}-E_{n+i,n+i},E_{i,j}]=[E_{i,i},E_{i,j}]-[E_{n+i,n+i},E_{i,j}]=E_{i,j}-0=E_{i,j}$$ We also have $$[H_{j},E_{i,j}]=[E_{j,j}-E_{n+j,n+j},E_{i,j}]=[E_{j,j},E_{i,j}]=-E_{i,j}$$
The same happens for matrix $E_{n+j,n+i}$ because $$[H_{i},E_{n+j,n+i}]=[E_{i,i}-E_{n+i,n+i},E_{n+j,n+i}]=[-E_{n+i,n+i},E_{n+j,n+i}]=-[E_{n+i,n+i},E_{n+j,n+i}]=E_{n+j,n+i}$$
We also claim that these elements vanish with adjoint action of $h_{k},k\not=i,j$. Thus $[H_{k},X_{i,j}]=(L_{i}-L_{j})(H_{k})X_{i,j}$, in other words $X_{i,j}$ is an eigenvector of $h$ with eigenvalue $L_{i}-L_{j}$.
\end{proof}

\begin{lemma}
The roots of $so_{6}$ can be identified as $\{\pm L_{i}\pm L_{j}\}$, $i\not=j$.  
\end{lemma}
\begin{proof}
We construct $Y_{i,j}=E_{i,n+j}+E_{j,n+i}$, $Z_{i,j}=E_{n+i,j}+E_{n+j,i}$ like above. We shall not do the same computations and conclude that they have eigenvalue $L_{i}+L_{j},-L_{i}-L_{j}$ respectively. The $i=j$ case would imply $X_{i,i}=H_{i}$ in $h$, while $Y_{i,j}=2E_{i,n+i},Z_{i,i}=2E_{n+i,i}$ not in $V$. Thus the roots of $so_{6}$ are just $\pm L_{i}\pm L_{j}\in h^{*}$, $i\not=j$. We thus have $4*(C^{n}_{2})=2n(n-1)$ roots. 

In the case $n=3$, the 12 roots are $\{\pm L_{1}\pm L_{2}\}, \{\pm L_{2}\pm L_{3}\}, \{\pm L_{3}\pm L_{2}\}$.
\end{proof}

\begin{lemma}
The Weyl group of $so_{6}$ has order 24 and is the semidirect proudct of $\Z/2\Z\oplus \Z/2\Z$ and $S_{3}$. It is isomorphic to $S_{4}$.
\end{lemma}
\begin{proof}
We wish to show that the group exact sequence $1\rightarrow \Z/2\Z\oplus \Z/2\Z \rightarrow W\rightarrow S_{3}\rightarrow 1$ holds. We shall prove the statment along the lines of Fulton and Harris.

The Weyl group is generated by reflection in the hyperplanes perpendicular to the roots $L_{i}\pm L_{j}$. We define $W_{\pm L_{i}\pm L_{j}}$ to be transformation orthogonal to $\pm L_{i}\pm L_{j}$.So $W_{L_{1}+L_{2}}$ maps $L_{1}$ to $-L_{2}$, $L_{2}$ to $-L_{1}$ , and keeps $L_{3}$ to the same position. $W_{L_{1}-L_{2}}$ exchanges $L_{1}$ and $L_{2}$, and $-L_{2}$ and $-L_{1}$. Thus combined together they change $L_{1},L_{2}$ to their negative and keep $L_{3}$. 

Since reflection by $W_{L_{1}-L_{2}}$ exchanges the two axes, the Weyl group contains a full subgroup isomorphic to $S_{3}$. The kernel of this action consists of transformations of determinant 1 that act as $-1$ on an even number of axis. They are generated by $W_{L_{1}+L_{2}}W_{L_{1}-L_{2}}$ and $W_{L_{2}+L_{3}}W_{L_{2}-L_{3}}$, both of order 2. Since $$W_{L_{1}+L_{2}}W_{L_{1}-L_{2}}W_{L_{2}+L_{3}}W_{L_{2}-L_{3}}=W_{L_{1}+L_{2}}W_{L_{1}-L_{2}}W_{L_{2}+L_{3}}W_{L_{2}-L_{3}}$$ we conclude that the group is abelian of order 4, thus isomorphic to $\Z/2\Z\oplus \Z/2\Z$. By definition of semidirect product this implies $W(so_{6})\cong \Z/2\Z\oplus \Z/2\Z \rtimes S_{3}$. 

To prove $\Z/2\Z\oplus \Z/2\Z \rtimes S_{3}\cong S_{4}$ we need to show the subgroup $$K=\{e,(12)(34),(13)(24),(14)(23)\}$$ is normal in $S_{4}$, $K\cap S_{3}=\{e\}$, and $S_{4}=KS_{3}$ as a group product. It is clear that $K$ is normal because we may decompose any element in $S_{4}$ as a product of non-intersecting transpositions, and the inner automorphism keeps the total transposition number be constant. It is also clear the $S_{3}\cap K=\{e\}$. The last statement is verified by straightforward computation. 
\end{proof}

\begin{remark}
We remark that the above method is totally analgous to the general $so_{n}$ case, where the Weyl group is $(\Z/2\Z)^{n-1}\rtimes S_{n}$. 
\end{remark}

\lemma We claim that $R(sl_{4})\cong R(so_{6})$.
\begin{proof}
This is because $Z[T_{sl_{4}}]\cong Z[T_{so_{6}}]$, and the Weyl group is isomorphic. Since $R(G)=R(T)^{W}$, we showed $R(sl_{4})\cong R(so_{6})$. Of course one may peek into the messy proof below to see the ring structure; the proof is essentialy nothing but change of variables (or coordinates). 
\end{proof}

\remark In fact this comes from the following more elementary fact:
\lemma $so_{6}\cong sl_{4}$. 
\begin{proof}
This is a proof borrowed from help in Mathoverflow, see \cite{BS}. I should avoid quoting verbatim for proof I do not fully understand and only noting that we have $Spin(6)\cong SU(4)$. We notice that $SU(4)$ is simply connected which acts on $\C^{4}$ preserving the hermitian form that "descend" down to $SO(6)$ because it preserves the real structure by conjugation. This implies $SU(4)\rightarrow SO(6)$ is a double covering. Thus $SU_{4}\cong Spin(6)$ and as a result $so_{6}\cong sl_{4}$. 
\end{proof}

\begin{theorem} $R(T_{so_{6}})$ is a free module over $R(so_{6})$.
\begin{proof}
This is the analgous consequence that $R(T_{sl_{4}})$ is a free module over $R(sl_{4})$. 
\end{proof}
\end{theorem}

We now write down the 'messy' proof because this provides an explicit basis;
\begin{lemma}
$R(h_{so_{6}})\cong \Z[x,y,z,x^{-1},y^{-1},z^{-1}, (xyz)^{1/2}]$.
\end{lemma}
\begin{proof} This is clear because the diagonal matrice in $so(6)$ satisfies $A+D^{t}=0$. Thus we can identify the top 3 elements as $a,b,c$, and the bottom 3 elements as their inverses. The last term exists because of the spin representation (my understanding of this is still not very clear). 
\end{proof}

\begin{lemma}

The representation ring of $so_{6}$ is isomorphic to\\ $\Z[x+y+z+x^{-1}+y^{-1}+z^{-1},x^{\frac{1}{2}}y^{\frac{1}{2}}z^{\frac{1}{2}}+x^{\frac{1}{2}}y^{-\frac{1}{2}}z^{-\frac{1}{2}}+x^{-\frac{1}{2}}y^{-\frac{1}{2}}z^{\frac{1}{2}}+x^{-\frac{1}{2}}y^{\frac{1}{2}}z^{-\frac{1}{2}}, x^{-\frac{1}{2}}y^{-\frac{1}{2}}z^{-\frac{1}{2}}+x^{\frac{1}{2}}y^{-\frac{1}{2}}z^{\frac{1}{2}}+x^{-\frac{1}{2}}y^{\frac{1}{2}}z^{\frac{1}{2}}+x^{\frac{1}{2}}y^{\frac{1}{2}}z^{-\frac{1}{2}}]$
\end{lemma}
\begin{proof}
This follows from the fundamental weights in $\Z[T]^{W}$. The later two are spin representations, and the former one is the standard representation. The interested reader should refer to \cite[p.~215]{FL}.
\end{proof}

\begin{theorem}
$\Z[x,y,z,x^{-1},y^{-1},z^{-1}, {xyz}^{1/2}]$ is a a free module over $\Z[x+y+z+x^{-1}+y^{-1}+z^{-1},x^{\frac{1}{2}}y^{\frac{1}{2}}z^{\frac{1}{2}}+x^{\frac{1}{2}}y^{-\frac{1}{2}}z^{-\frac{1}{2}}+x^{-\frac{1}{2}}y^{-\frac{1}{2}}z^{\frac{1}{2}}+x^{-\frac{1}{2}}y^{\frac{1}{2}}z^{-\frac{1}{2}},x^{-\frac{1}{2}}y^{-\frac{1}{2}}z^{-\frac{1}{2}}+x^{\frac{1}{2}}y^{-\frac{1}{2}}z^{\frac{1}{2}}+x^{-\frac{1}{2}}y^{\frac{1}{2}}z^{\frac{1}{2}}+x^{\frac{1}{2}}y^{\frac{1}{2}}z^{-\frac{1}{2}}]$
\end{theorem}

\begin{proof}
We notice the isomorphism under the change of variables $x^{\frac{1}{2}}y^{\frac{1}{2}}z^{\frac{1}{2}}=a,x^{\frac{1}{2}}y^{-\frac{1}{2}}z^{-\frac{1}{2}}=b,x^{-\frac{1}{2}}y^{-\frac{1}{2}}z^{\frac{1}{2}}=c, x^{-\frac{1}{2}}y^{\frac{1}{2}}z^{-\frac{1}{2}}=d$, we may translate the statement to the following: $B=\Z[ab,cd,da,bd,ca,bc]$ is free over $A=\Z[a+b+c+d,ab+cd+da+bd+ca+bc,abc+bcd+bda+acd]$ with $abcd=1$. 

The proof would be immediate if we can show $B=\Z[a,b,c,d]$, for this would imply we are dealing with the $sl_{4}$ case with a change of variables. But this is precisely the case because the base ring contains $(xyz)^{1/2}=a$. This, of course is just a verification that $so_{6}\cong sl_{4}$. 
\end{proof}

\remark We thus concluded that the basis elements are: (in here the $x,y,z$ are respective dual weights of $h_{i}$ in $sl_{4}$ and $so_{6}$ in case there is any confusion)\\
For the $sl_{4}$ case:
$$\{1,x,y	z,x^{2},xy,xz,y^{2},x^{3},x^{2}y,x^{2}z,xyz,xy^{2},y^{2}z,x^{3}y,x^{3}z,x^{2}y^{2},x^{2}yz,	xy^{2}z	,x^{3}y^{2},x^{2}y^{2}z,x^{3}y^{2}z\}$$
For the $so_{6}$ case:

$$\{1, \sqrt{xyz}, \frac{1}{y}, xyz, x, z, \frac{x}{yz}, x^{3/2}y^{3/2}z^{3/2}, x^{3/2}(yz)^{1/2}, (xy)^{1/2}z^{3/2}, \frac{(xz)^{1/2}}{y^{1/2}}, \frac{x^{1/2}}{y^{3/2}z^{1/2}}, x^{2}yz, xyz^{2}, x^{2}, xz\}$$ and $$\{\frac{x}{y}, x^{5/2}y^{1/2}z^{1/2}, \frac{x^{3/2}z^{1/2}}{y^{1/2}}, x^{2}z\}$$ because we cannot put them in one line.

\chapter{ Topics for further study}
\discussion We would like to discuss the following topics:\\
\indent The situation for $sp_{n}$, $spin_{n}$, $so_{2n+1}$ is largely unclear to me. Not only that, but the case for $so_{2n}$ is more complicated than expected because in general we are expecting a basis of size $2^{n-1}n!$. So to work on the $so_{8}$ case we would expect $2^{3}4!=192$ elements and impossible to verify by hand. 

\indent Even if we propose a tentative basis, I think I will need computer programming to verify (by index theory) that the basis we suggested is indeed a basis. It is not apparent to me that using the inner product defined by index theory, we would yield an unimodular matrix. Worse, we do not know how to find a tentative basis in general because the minimal polynomial trick we used for $sl_{n}$ (and successfully employed in $so_{2n}$ is largely inapplicable to the remaining $B,C$ cases and $D$ in general. 

\indent Thus the general proof is still unclear. We could have adopted a geometric approach to this problem by manipulating the weight vectors to see how they cancel out with each other. But in high dimensions (like $so_{8}$), we can no longer do so (manipulating 192 vectors would be too much work). One may try to work out Steinberg's basis for each specific case by programming, but it remains to be discovered the general pattern in individual cases. 

\indent We should acknowledge some progress made on other fronts that may help to solve this problem. The main method is divided difference operators and Schubert polynomials in combinatorics, which the author knows very little. In 1996 Sara c. Billey published a paper \cite{Biley} that related to this problem by working on Konstant Polynomials over $G/B$, in here $B$ is Borel subgroup. Earlier William Fulton published a related paper \cite{Fulton} on the degenercy loci of $G/B$ viewed as an algebraic variety. Recently I found John Jiang \cite{Jiang}  from Stanford is working on a related topic using Markov chains and Macdonlad operators. All of these are outside of the scope of this paper, but it seems likely that this small problem of computing the basis of $R(T)$ over $R(G)$ may have association with deeper problems in other fields. 

\indent It would also be curious to know if we can treat this in a more general setting. I know that Yuri. Manin suggested to investigate the cohomology ring of $G/B$ (or $G/T$) in the 1970s. But I do not know how this associates with the problem at hand. Still this line of thinking using the analogous relationship between K-Theory and cohomology may be helpful to understand the structure of $G/T$ and $R(T)/R(G)$ as we know the bi-invariant differential forms are associated with $H^{*}(G/T)$. But these are purely speculation at this point. 

\indent Therefore we conclude that Steinberg's basis is still practical for type $B,C,D, E,F,G$. The author wish to know (or else this paper would not be written ) of the situation in other types because he could not find much relevant information in the reference available to him.

\singlespace

\end{document}